\documentclass[twoside]{birkmult}

\usepackage{amsmath, amsthm, amssymb, mathrsfs}
\usepackage{setspace}
\usepackage{epsfig, graphicx}

\newcommand{\T}    {\mathbb{T}}
\newcommand{\D}    {\mathbb{D}}
\newcommand{\R}    {\mathbb{R}}
\newcommand{\C}    {\mathbb{C}}
\newcommand{\N}    {\mathbb{N}}
\newcommand{\Z}    {\mathbb{Z}}

\newcommand{\Pj}   {\mathbf{P}}

\newcommand{\CoM}  {\mathbf{BVT}}
\newcommand{\pma}  {\mathbf{MA}}

\newcommand{\W}    {\mathcal{W}}
\newcommand{\Pc}   {\mathcal{P}}
\newcommand{\Mc}   {\mathcal{M}}
\newcommand{\Rc}   {\mathcal{R}}
\newcommand{\K}    {\mathcal{K}}

\newcommand{\supp} {\textnormal{supp}}
\newcommand{\cp}   {\textnormal{cap}}
\newcommand{\dist} {\textnormal{dist}}
\newcommand{\diam} {\textnormal{diam}}
\newcommand{\Arg}  {\textnormal{Arg}}
\newcommand{\ang}  {\textnormal{Angle}}
\newcommand{\re}   {\textnormal{Re}}

\newcommand{\Log}  {\textnormal{Log}}

\newcommand{\spp}  {S}
\newcommand{\rspp} {\spp^\prime}
\newcommand{\wspp} {\widetilde\spp}
\newcommand{\mes}  {\lambda}
\newcommand{\rmes} {\lambda^\prime}
\newcommand{\wmes} {\widetilde\lambda}

\newcommand{\Lm}   {\Lambda}
\newcommand{\arm}  {\varphi}
\newcommand{\angs} {\theta}

\newcommand{\cws}  {\stackrel{*}{\rightarrow}}
\newcommand{\cic}  {\stackrel{\scriptsize\cp}{\rightarrow}}
\newcommand{\ged}  {\mu_{(\spp,\T)}}

\newtheorem{theorem}{Theorem}[section]
\newtheorem{lemma}[theorem]{Lemma}
\newtheorem{corollary}[theorem]{Corollary}

\numberwithin{equation}{section}

\begin{document}

\title[Meromorphic Approximants to Cauchy Transforms]{Meromorphic Approximants to Complex Cauchy Transforms with Polar Singularities}

\author{L.Baratchart}

\address{INRIA, Project APICS \\
2004 route des Lucioles --- BP 93 \\
06902 Sophia-Antipolis, France}

\email{laurent.baratchart@sophia.inria.fr}

\author{M.Yattselev}

\address{INRIA, Project APICS \\
2004 route des Lucioles --- BP 93 \\
06902 Sophia-Antipolis, France}

\email{myattsel@sophia.inria.fr}

\date{\normalsize \today}

\begin{abstract}
We study AAK-type meromorphic approximants to functions of the form
$$F(z) = \int\frac{d\lambda(t)}{z-t}+R(z),$$
where $R$ is a rational function and $\lambda$ is a complex measure with compact regular support included in $(-1,1)$, whose argument has bounded variation on the support. The approximation is understood in $L^p$-norm of the unit circle, $p\geq2$. We dwell on the fact that the denominators of such  approximants satisfy certain non-Hermitian orthogonal relations with varying weights. They resemble the orthogonality relations that arise in the study of  multipoint Pad\'e approximants. However, the varying part of the weight implicitly depends on the orthogonal polynomials themselves, which constitutes the main novelty and the main difficulty of the undertaken analysis. We obtain that the counting measures of poles of the approximants converge to the Green equilibrium distribution on the support of $\lambda$ relative to the unit disk, that the approximants themselves converge in capacity to $F$, and that the poles of $R$ attract at least as many poles of the approximants as their multiplicity and not much more. 
\end{abstract}



\maketitle

\section{Introduction}
\label{sec:intro}

This paper is concerned with the asymptotic behavior of certain meromorphic
 approximants to functions of the form
\begin{equation}
\label{fH}
F(z)=\int\frac{d\mes(t)}{z-t}+R(z),
\end{equation}
where $R$ is a rational function, holomorphic at infinity, and
$\mes$ is a complex measure compactly and regularly supported on $(-1,1)$. 

The meromorphic approximants that we consider
are optimal, for fixed number of poles in the unit disk, 
with respect to an $L^p$-norm on 
the unit circle. When studying them, we assume
that $\supp(\mes)$ and all poles of $R$ lie in the open unit disk, 
so that $F$ is indeed $p$-summable on the unit circle. 
The asymptotics are then understood 
when the number of poles grows large. In the case where $p=\infty$, this type 
of approximant was introduced by V. M. Adamyan, D. Z. Arov, and M. G. Krein 
in their famous paper
\cite{AAK71}. Here we deal with their natural generalization to $L^p$,
although we restrict ourselves to the range $2\leq p\leq\infty$ for technical
reasons to be explained later. The meromorphic approximation problem also has
a conformally invariant formulation on Jordan domains with rectifiable
boundary, to which the results of the present paper transpose with obvious
modifications
if $\supp(\mes)$ is contained in a closed hyperbolic geodesic arc rather
than a segment. The interested
reader will have no difficulty to carry out this generalization using the
construction of \cite[Sec. 5]{BMSW06}.

The study of best meromorphic approximants is quite recent. After the development
of the Adamyan-Arov-Krein theory for $p=\infty$ in \cite{AAK71}, the latter was 
extended to the range $1\leq p<\infty$ by 
F. Seyfert and the first author in \cite{BS02}, 
and independently by V. A. Prokhorov in \cite{Pr02}.
However, it is only for $p\geq2$ that the authors of \cite{BS02} were able to 
express the error in terms of (generalized) singular vectors of a Hankel
operator and subsequently to obtain integral formulas for that error 
when the approximated function is represented as a Cauchy integral.
These formulas make connection with non-Hermitian orthogonality, and form the
basis of the present approach. 

The AAK theory had considerable impact in rational approximation,
for on retaining only the principal part of a best meromorphic approximant 
to a function analytic outside the disk and sufficiently smooth on the circle,
one obtains a near-best rational approximant  to that function
\cite{Glov84}.
This is instrumental in
Parfenov's solution, for simply connected domains,
to the Gonchar conjecture \cite{Gon78b}
on the degree of rational approximation to holomorphic 
functions on compact subsets of their domain of analyticity\footnote{
the proof of this conjecture
was later carried over to the multiply connected case by Prokhorov in 
\cite{Pr93a}.}, and also
for instance in Peller's converse theorems
on smoothness of functions from their error rates in rational approximation
\cite{Pel80,Pel83}. The same principle is also at work in the articles
\cite{Br87, And94} that deal with rational approximation to Markov
functions, that is, functions of the form (\ref{fH}) where $R\equiv0$ and
$\lambda$ is positive. 

Another connection between meromorphic approximation and rational 
approximation that ought to be mentioned occurs when $p=2$. In this case,
a best meromorphic approximant to a function
analytic outside the unit disk is in fact rational, and 
turns out to be a special type
of multipoint Pad\'e approximant that interpolates the function
with order 2 at the reflections  of its poles across the unit circle
\cite{Lev69, BSW96}. Of course the interpolation points are not known 
{\it a priori} which accounts for the nonlinearity of the problem.
Nevertheless, after the work in \cite{GL78},  
this connexion was used
in \cite{BStW01} to establish the convergence rate
of best  $L^2$ rational approximants to Markov functions.
The forthcoming results will, in particular, generalize these results to a 
larger class of functions.

Best meromorphic approximants to Markov functions were studied {\it per se} by E. B. Saff, V. Prokhorov and the first author
in \cite{BPS01a}. Using results from \cite{And94} to make connection with
orthogonality, these authors prove (and give error rates for)
the uniform convergence of such
approximants, locally uniformly on $\overline{\C}\setminus I$, 
whenever $1\leq p\leq\infty$ provided that $\mes$ satisfies the Szeg\H{o}
condition: $\log d\mes/dt\in L^1(I)$. 

The present paper appears to be the first to deal with convergence of best meromorphic approximants to general functions of the form (\ref{fH}) in the case where $\mes$ is a complex measure (the special case when $F$ has two branchpoints and no poles is treated in \cite[Sec. 10]{BS02}). 

This paper is organized as follows. Section \ref{sec:merom} deals with meromorphic 
approximants that are solutions (more generally: critical points) of the meromorphic approximation problem for functions of the form (\ref{fH}). We discuss the asymptotics of poles as being the limit zero distribution of polynomials satisfying certain non-Hermitian orthogonality relations with respect to varying measures. We apply the results to the convergence in capacity of these approximants, and to the convergence of some of their poles to the polar singularities of $F$. All the proofs are presented in Section \ref{sec:proofs}. Some computational results are adduced in Section \ref{sec:numer} and the Appendix contains necessary material and notation from potential theory that we use throughout the paper.

Finally, we mention that all the results below have their counterpart in diagonal multipoint Pad\'e approximation, where they allow more irregular $\lambda$ than can usually be handled via classical results (see \cite{GRakh87}). We do not include this to keep the size of the paper within reasonable bounds, and refer the interested reader to \cite[Ch. III]{thYat} or \cite{uBY2}.

\section{Meromorphic Approximation}
\label{sec:merom}

Let $\mes$ be a complex Borel measure whose support
$\spp:=\supp(\mes)\subset(-1,1)$ consists of infinitely many points. 
Denote by $|\mes|$ the total variation measure. 
Clearly $\mes$ is absolutely continuous with respect to $|\mes|$, and we 
shall assume that its Radon-Nikodym derivative (which is of unit
modulus $|\mes|$-a.e.) is of bounded variation. 
In other words, $\mes$ is of the form
\begin{equation}
\label{eq:mesDecomp}
d\mes(t)=e^{i\arm(t)}d|\mes|(t),
\end{equation}
for some real-valued argument function $\arm$ such that\footnote{Note that $e^{i\arm}$ has
  bounded variation if and only if $\arm$ can be chosen of bounded
  variation.}
\begin{equation}
\label{boundedvarphi}
V(\arm,\spp):=\sup\left\{\sum_{j=1}^N|\arm(x_j)-\arm(x_{j-1})|\right\}<\infty,
\end{equation}
where the supremum is taken over all finite sequences $x_0<x_1<\ldots<x_N$ in $\spp$ as $N$ ranges over $\N$.  

For convenience, we extend the definition of $\arm$ to the whole of $\R$ as
follows. Let $I:=[a,b]$ be the convex hull of $\spp$. It is easy to see that
if we interpolate $\arm$ linearly in each component of $I\setminus\spp$ and if
we set $\arm(x):=\lim_{t\to a, \; t\in\spp}\arm(t)$ for $x<a$ and
$\arm(x):=\lim_{t\to b, \; t\in\spp}\arm(t)$ for $x>b$ (the limits exist by
(\ref{boundedvarphi})),  the variation of $\arm$ will remain the same. In
other words, we may arrange things so that the extension of $\arm$, still
denoted by $\arm$, satisfies
$$V(\arm,\spp)=V(\arm,\R)=:V(\arm).$$

Among all complex Borel measures of type  (\ref{eq:mesDecomp})-(\ref{boundedvarphi}), we shall consider only a subclass $\CoM$ defined as follows. We say that a complex measure $\mes$, supported on $(-1,1)$, belongs to the class $\CoM$ if
\begin{itemize}
        \item[(1)] {\it $\supp(\mes)$ is a regular set};
        \item[(2)] {\it there exist positive constants $c$ and $L$ such that, for any $x\in\supp(\mes)$ and $\delta\in(0,1)$,  the total variation of $\mu$ satisfies $|\mes|([x-\delta,x+\delta])\geq c\delta^L$};
        \item[(3)] {\it $\mes$ has an argument of bounded variation}.
\end{itemize}

Denote by $\Pc_n$ the space of algebraic polynomials of degree at most $n$ and
by $\Mc_n$ the subset consisting of monic polynomials of degree $n$ 
whose zeros lie in the open unit disk, $\D$. 

Define
\begin{equation}
\label{eq:mainFun}
F(z):=\int\frac{d\mes(\xi)}{z-\xi}+R_s(z),
\end{equation}
with $\mes\in\CoM$ and $R_s\in\Rc_{s-1,s}$, where
$$\Rc_{m,n}:=\{p_m/q_n: \; p_m\in\Pc_m, \; q_n\in\Mc_n\}$$
is the set of rational functions of type $(m,n)$ with all their poles in 
$\D$. Hereafter we shall denote by $Q_s$ the denominator of $R_s$, assumed to
be in irreducible form, which is a monic polynomial with zeros in $\D$ of the
form 
\begin{equation}
\label{defQs}
Q_s(z)=\prod_{\eta\in\rspp}(z-\eta)^{m(\eta)},
\end{equation}
where $\rspp$ is the set of poles of $R_s$ and $m(\eta)$ stands for the
multiplicity of $\eta\in\rspp$. Thus, $F$ is a meromorphic function in
$\C\setminus\spp$ with poles at each point of $\rspp$ and therefore holomorphic in $\overline\C\setminus\wspp$, where
$$\wspp:=\spp\cup\rspp.$$ 
Note in passing that $F$ does not reduce to a rational function since $\spp$ 
consists of infinitely many points, cf. \cite[Sec. 5.1]{BMSW06} for a 
detailed argument.

In this paper we consider the behavior of certain meromorphic (AAK-type)
 approximants to a function $F$ of the form 
(\ref{eq:mainFun}). Of particular importance to us will be the asymptotic
behavior of the poles of the above-mentioned approximants. The latter will be
quantified in terms of the weak$^*$ convergence,
when the number of poles increases indefinitely, of 
the \emph{counting measures} of these poles. By definition, the counting
measure of the poles of a meromorphic function is
the discrete probability measure with equal mass at each finite pole, counting 
multiplicities. The weak$^*$ convergence is understood in the usual sense 
where measures, endowed with the norm of total variation,
are regarded as the dual space to continuous functions with compact support.

We denote by $H^p$, $p\in[1,\infty]$, the {\it Hardy space} of the unit disk consisting of holomorphic functions $f$ such that
\begin{equation}
\label{eq:defNorm}
\begin{array}{lll}
\displaystyle \|f\|_p^p := \sup_{0<r<1}\frac{1}{2\pi}\int_{\T}|f(r\xi)|^p|d\xi|<\infty  & \mbox{ if } & p\in[1,\infty), \smallskip \\
\displaystyle \|f\|_{\infty}:=\sup_{z\in\D}|f(z)|<\infty  & \mbox{ if } & p=\infty.
\end{array}
\end{equation}
It is known (\cite[Thm. I.5.3]{Garnett}) that a function in $H^p$ is uniquely determined by its trace (nontangential limit) on the unit circle, $\T$, and that the $L^p$-norm of this trace is equal to the $H^p$-norm of the function, where $L^p$ is the space of $p$-summable functions on $\T$. This
way $H^p$ can be regarded as a closed subspace of $L^p$.   Analogously, we define $\bar H_0^p$, $p\in[1,\infty]$, consisting of holomorphic functions in $\overline\C\setminus\overline\D$ that vanish at infinity and satisfy (\ref{eq:defNorm}) this time with $1<r<\infty$.

Now, the meromorphic approximants that we deal with are defined as
follows. For $p\in[1,\infty]$ and $n\in\N$, the 
{\it class of meromorphic functions of degree} $n$ in $L^p$ is 
\begin{equation}
\label{alternativemerodef}
H^p_n:=H^p+\Rc_{n-1,n} = H^pB_n^{-1},
\end{equation}
which is a closed subset of $L^p$ (it is in fact weakly closed if 
$1<p<\infty$ and weak$^*$ closed if $p=\infty$, see \cite[Lemma 5.1.]{BS02}). In (\ref{alternativemerodef}) we denote by $B_n$ the {\it set of Blaschke products of degree at most} $n$, consisting of rational functions of the form
$$b(z)=e^{ic}\frac{q(z)}{\widetilde q(z)}, \;\;\; q\in\Mc_k, \;\;\; \mbox{where} \;\;\; \widetilde q(z) := z^k\overline{q(1/\bar z)}, \;\;\; k\leq n.$$
We shall call $\widetilde q$ the {\it reciprocal polynomial} of $q$ in $\Pc_k$. We also say that $b$ is {\it normalized} if $e^{ic}=1$. Thus, the members of $B_n$ are rational functions of degree at most $n$
holomorphic in $\D$ and having modulus $1$ everywhere on $\T$. 

Our {\it best-$L^p$ meromorphic approximation problem} can now be stated as follows.
\newline
\newline
{\it $\pma(p):$ Given $p\in[1,\infty]$, $f\in L^p$,  and $n\in\N$, find $g_n\in
  H^p_n$ such that}
\begin{equation}
\label{eq:merApprProblem}
\|f-g_n\|_p=\inf_{g\in H^p_n}\|f-g\|_p.
\end{equation}

Originally this problem was solved for the case $p=\infty$ by V. M. Adamyan,
D. Z. Arov, and M. G. Krein in \cite{AAK71} and the solution came through
operator theory. The most accessible reference to this result is perhaps
\cite{Young}. Later F. Seyfert and the first author \cite{BS02}, and independently 
V. A. Prokhorov \cite{Pr02}, generalized it to the case $1\leq p\leq\infty$,
but it is only in \cite{BS02} and when $p\in[2,\infty]$ that concrete
equations were obtained for the approximants. These form the basis of our
approach, and presently limit its scope to $p\geq2$.

This solution of $\pma(\infty)$ is known to be unique, provided that $f$ belongs to the \emph{Douglas algebra} $H^\infty+C(\T)$, where $C(\T)$ denotes the space of continuous 
functions on $\T$ \cite{AAK71}. In particular, the solution to $\pma(\infty)$ 
is unique when $f$ is of type (\ref{eq:mainFun}) since the latter
is analytic in some neighborhood of the unit circle ({\it i.e.} in the
complement of $\wspp$). When $p<\infty$, a solution needs not be unique even
if $f$ is very smooth \cite[Sec. 5]{BS02}. Therefore, when making a 
statement about a sequence $\{g_n\}$ of solutions to $\pma(p)$, 
it is understood that a particular solution 
has been selected for each $n$ and that the statement holds true
regardless the selection. 

Given $F$ as in (\ref{eq:mainFun}), we shall be interested
in three types of questions:
\begin{itemize}
        \item[(a)] {\it What is the asymptotic distribution of the poles of
          best-$L^p$ meromorphic approximants to $F$
          as $n$ tends to $\infty$?}
        \item[(b)] {\it Do some of these poles converge to the polar singularities of $F$?} 
        \item[(c)] {\it What can be said about the convergence of such approximants to $F$?}
\end{itemize}

As it is the case of interest here, we shall restrict our discussion to 
the situation where the approximated
function is of the form (\ref{eq:mainFun}), and accordingly we write $F$ 
instead of $f$.
We should note that $\pma(2)$ reduces to rational 
approximation. Indeed,  $L^2$ can be 
decomposed into the orthogonal sum of $H^2$ and its orthogonal complement 
$\bar H^2_0,$ which consists of analytic functions in $\C\setminus\overline\D$
vanishing at $\infty$ with norm $\sup_{r>1}\|f(r\cdot)\|_2$ (compare
(\ref{eq:defNorm})).  
Now, since $F$ is the Cauchy transform of a measure supported in $\D$, it
belongs to $\bar H^2_0$. Thus for any $g=(h+p_{n-1}/q_n)\in H^2_n$ with $h\in
H^2$ and $p_{n-1}/q_n$ a rational function in $\bar H^2_0$, we get by 
orthogonality
$$\|F-g\|_2^2=\|h\|_2^2+\|F-p_{n-1}/q_n\|_2^2.$$
Clearly then, for $g$ to be a best approximant $h$ must be zero. It also turns out in this case that best approximants interpolate $F$ with order 2 at the reflections of their poles across $\T$ and also at 
infinity with order 1. Thus, one can regard $\pma(2)$ as an interpolation problem of the multipoint Pad\'e type where the interpolation points are \emph{implicitly defined} by the solution. Despite this, we shall not distinguish $p=2$ from the other cases but rather keep a unified operator approach.

Let us denote by $\Pj_+$ and $\Pj_-$ the analytic and 
anti-analytic projections acting on Fourier series as
$$\Pj_+\left(\sum_{k=-\infty}^{+\infty}a_ke^{ik\theta}\right)=
\sum_{k=0}^{+\infty}a_ke^{ik\theta}~~~~~~~~
\Pj_-\left(\sum_{k=-\infty}^{+\infty}a_ke^{ik\theta}\right)=
\sum_{k=-\infty}^{-1}a_ke^{ik\theta}.$$
By a well-known theorem of M. Riesz \cite{Garnett}, ${\bf P}_+:L^p\to H^p$ and
${\bf P}_-:L^p\to {\bar H}_0^p$ are bounded when $1<p<\infty$, and
when $p=2$ they are just the
orthogonal projections associated to the orthogonal decomposition: 
$L^2=H^2\oplus{\bar H}_0^2$. When $f\in L^1$, we simply
regard  ${\bf P}_+(f)$ and ${\bf P}_-(f)$ as Fourier 
series of distributions.

For each $p\in[2,\infty]$, the {\it Hankel operator with symbol} $F\in L^p$ is given by
$$\begin{array}{rll}
A_F:H^{p^\prime} & \to     & \bar H^2_0 \\
u                & \mapsto & \Pj_-(Fu), 
\end{array}$$
where $p^\prime$ is conjugate to $p$ modulo 2, 
i.e. $1/p+1/p^\prime=1/2$. 

For $n=0,1,2,\dots$, the {\it $n$-th singular number} of the operator 
$A_F$ is defined to be 
$$\sigma_n(A_F):=\inf\left\{|||A_F-\Gamma|||, \;\;\; \Gamma:H^{p^\prime}\to \bar H^2_0 \mbox{ a linear operator of rank } \leq n\right\},$$
where $|||\cdot|||$ stands for the operator norm;
when $p=2$ we assume in addition that $\Gamma$ is weak$^*$ continuous. 

Note that if $p=\infty$ then $p^\prime=2$, hence $A_F$ operates 
between  Hilbert spaces, and since it is compact\footnote{This can be deduced
  from the fact that $F\in C(\T)$, see \cite[Thm. I.5.5]{Peller}.}  
the $\sigma_n(A_F)$ are just the singular values of $A_F$, that is,
the square-roots of the eigenvalues of $A^*_FA_F$ arranged in nonincreasing 
order; throughout $A^*_F$ indicates the adjoint of $A_F$. When $2\leq
p<\infty$, the usual eigenvector equation gets replaced by a \emph{nonlinear}
equation of Hammerstein type. More precisely, to each 
$\sigma_n(A_F)$ there exists (at least one) 
$v_n\in H^{p^\prime}$ of unit norm, whose inner factor\footnote{Recall that any function $h$ from a Hardy space can be written as the product of an inner and an outer factors. The outer factor of $h$ is equal to the exponential of the Riesz-Herglotz transform of $\log|h|$, and the inner factor, which is a $H^\infty$ function unimodular a.e. on $\T$, consists of a Blaschke product (finite or infinite) and a singular inner factor which is the exponential of the Riesz-Herglotz transform of a singular measure on the unit circle. More on the inner-outer factorization of $H^p$ functions can be found, for example, in \cite{Garnett}.} is a Blaschke product
of degree at most $n$, such that 
\begin{equation}
\label{singvn}
\begin{array}{ll}
A_F^*A_F(v_n)=\sigma_n^2(A_F)\Pj_+\left(|v_n|^{p^\prime-2}v_n\right) & ~\mbox{if } ~p>2, \smallskip \\
A_F^*A_F(v_n)=\Pj_+\left(|A_F(v_n)|^2v_n\right)~~~{\rm and}~~~
\|A_F(v_n)\|_2=\sigma_n(A_F) & ~\mbox{if } ~p=2.
\end{array}
\end{equation}
Such a $v_n$ will be called a $n$-th singular vector for $A_F$.
From the definition it follows that
a $n$-th singular vector can be factored as
\begin{equation}
v_n=b_nw_n
\label{eq:innerOuterFactor}
\end{equation}
where $b_n\in B_n$ and $w_n$ is an 
outer function of unit norm in $H^{p^\prime}$. Actually, upon pairing with $v_n$,
equation (\ref{singvn}) implies that $v_n\in B_n$ if $p=2$, so that $w_n\equiv1$ in
this case. 

The solutions to $\pma(p)$ turn out to be exactly the functions of 
the form (\cite[Thm. 8.2]{BS02})
\begin{equation}
\label{eq:operatorBestAppr}
g_n=F-\frac{A_F(v_n)}{v_n}=\frac{\Pj_+(Fv_n)}{v_n},
\end{equation}
where $v_n$ is some $n$-th singular vector for $A_F$. Moreover, we have that
$\|F-g_n\|_p=\sigma_n(A_F)$. Such a $g_n$ is called a best meromorphic 
$L^p$-approximant of order $n$ to $F$.

The notion of a best approximant can be further weakened to the notion of 
a {\it critical point}. By definition, a function $g_n$ is a critical point of
order $n$ in $\pma(p)$ if and only if it assumes the form
\begin{equation}
\label{eq:DefCrPt}
g_n=F-\frac{A_F(v_n)}{v_n},
\end{equation}
where $v_n$ is a $H^{p^\prime}$ function of unit norm (a Blaschke product if
$p=2$) whose inner factor lies in $B_n\setminus B_{n-1}$ and
which is such that
\begin{equation}
\label{eq:criticalPoint}
\begin{array}{ll}
A_F^*A_F(v_n)=\gamma_n\Pj_+\left(|v_n|^{p^\prime-2}v_n\right), \;\;\; \gamma_n\in\R, & ~\mbox{if } ~p>2, \smallskip \\
A_F^*A_F(v_n)=\Pj_+\left(|A_F(v_n)|^2v_n\right) & ~\mbox{if } ~p=2.
\end{array}
\end{equation}
The difference with (\ref{singvn}) is that here $\gamma_n$ (or $\|A_F(v_n)\|_2$ if $p=2$) needs not be equal to $\sigma_n(A_F)$. With a slight abuse of language, we will continue to say that $v_n$ is a singular vector associated to $g_n$
although $\gamma_n$ may no longer be a singular value.
Note that, as for best meromorphic approximants, $v_n$ reduces to a Blaschke 
product when $p=2$. Thus $v_n$ has an inner-outer factorization of 
type (\ref{eq:innerOuterFactor}) where $b_n$ has exact degree $n$
and $w_n\equiv1$ if $p=2$.

Although their definition is a little technical, critical points are just those
$g_n\in H_n^p$ for which the derivative of $\|F-g_n\|_p$ with respect to 
$b_n\in B_n$ and $h\in H^p$ in factorization (\ref{alternativemerodef})
does vanish \cite{BS02}.
Beyond best approximants, the most important critical points are
{\it local best approximant} \cite[Prop. 9.3.]{BS02}. 
By definition, a local best approximant is some $g_n\in H^p_n$ for which 
there exists $\delta>0$ such that
$$g\in H_n^p ~~~ \mbox{ and } ~~~ \|g-g_n\|_p\leq\delta ~~~ 
\mbox{ imply } ~~~ \|F-g_n\|_p\leq\|F-g\|_p.$$
The reason why we introduce critical points is that all a numerical search
can yield in general is a local best approximant, and we feel it is
important that our results should apply to computable objects.

When $p\in[2,\infty)$, best and local best approximants have exactly $n$ poles,
counting multiplicities, hence they are {\it a fortiori} critical points of 
order $n$ (\cite[Prop. 9.2 and 9.3]{BS02}); such critical points are 
called \emph{irreducible}.
For $p=\infty$ the critical points are just the best meromorphic approximants, which are unique for each fixed $n$.  
So the notion is nothing new, but it may happen that a best approximant 
out of $H_n^\infty$ has less than $n$ poles. However, each time the number of poles of $g_n$ increases with $n$, it jumps to the maximum value $n$, in particular, there exists a subsequence of natural numbers, say $\N_0=\N_0(F)$, such that for 
each $n\in\N_0$ the best approximant $g_n$ has exactly $n$ poles in $\D$, 
{\it i.e.} it is irreducible 
(\cite[p. 114]{BS02}). Since the behavior of the poles of best approximants
from $H_n^\infty$ is entirely characterized by this subsequence, hereafter 
we say ``a sequence of irreducible critical points of order 
$n$'' to mean if $p=\infty$ that we pass to a subsequence if needed. 

The three theorems stated below constitute the main results of the paper. For the definitions of capacity, Green equilibrium distribution, and condenser (Green) capacity, the reader may want to consult the appendix.

\begin{theorem}
\label{thm:polesCriticalPoints}
Let $p\in[2,\infty]$, $p^\prime$ the conjugate exponent modulo 2, and 
$\{g_n\}_{n\in\N}$ be a sequence of irreducible critical points of order $n$ 
of $\pma(p)$ for $F$, where $F$ is given by (\ref{eq:mainFun})-(\ref{defQs}) with 
$\mes\in\CoM$. Then the counting measures of the poles of $g_n$ converge to $\ged$, the Green equilibrium distribution on $\spp$ relative to $\D$, in the weak$^*$ sense.
\end{theorem}

The previous theorem gives one answer to question $(a)$. The next one addresses question $(c)$ by stating that approximants behave rather nicely toward the approximated function, namely they converge {\it in capacity}\footnote{See the appendix for the definition of convergence in capacity.} to $F$ on $\D\setminus\spp$, and in the case $p=2$ uniformly in $\overline\C\setminus\D$. Moreover, $n$-th root estimates for the error are provided.

\begin{theorem}
\label{thm:convergenceCapacity}
Let $F$ and $\{g_n\}_{n\in\N}$ be as in Theorem \ref{thm:polesCriticalPoints}. Then 
\begin{equation}
\label{eq:Convergence1}
|(F-g_n)(z)|^{1/2n} \cic \exp\left\{U_\D^{\ged}(z)-\frac{1}{\cp(\spp,\T)}\right\}
\end{equation}
on compact subsets of $\D\setminus\spp$, where $U_D^{\ged}$ is the Green potential of $\ged$ relative to $\D$ and $\cic$ denotes convergence in capacity. In addition, in the case of rational and AAK approximation (i.e. when $p=2$ and $p=\infty$, respectively) it holds that
\begin{equation}
\label{eq:Convergence2}
|(F-g_n)(z)|^{1/2n} \cic \exp\left\{-\frac{1}{\cp(\spp,\T)}-U_\D^{\ged}(1/\bar z)\right\}
\end{equation}
 on closed subsets of $\overline\C\setminus(\D\cup\spp^*)$. Moreover, for $p=2$, it also holds that
\begin{equation}
\label{eq:Convergence3}
\limsup_{n\to\infty} |(F-g_n)(z)|^{1/2n} \leq \exp\left\{-\frac{1}{\cp(\spp,\T)}-U_\D^{\ged}(1/\bar z)\right\}
\end{equation}
uniformly in $\overline\C\setminus\D$, where $\spp^*$ is the reflection across $\T$ of $\spp$.
\end{theorem}

As a consequence of the previous theorem, we can prove a result on best rational approximation function of the type (\ref{eq:mainFun}) in $\overline\C\setminus\D$ that is classical in scope  (cf. \cite[Thm. 1$^\prime$]{GRakh87}), but new if $\mes$ vanishes on a subset of positive capacity of the convex hull of $\spp$. In what follows $\|\cdot\|_K$ stands for the supremum norm on a set $K$.

\begin{corollary}
\label{cor:BestRational}
Let $F$ be given by (\ref{eq:mainFun})-(\ref{defQs}) with $\mes\in\CoM$ and ${\mathbb E}:=\overline\C\setminus\D$. Then
\begin{equation}
\label{eq:BestRational}
\lim_{n\to\infty} \rho_n(F,{\mathbb E})^{1/2n} = \exp\left\{-\frac{1}{\cp(\supp(\mes),\T)}\right\},
\end{equation}
where
\[
\rho_n(F,{\mathbb E}) := \inf_{r\in\Rc_{n,n}} \|F-r\|_{\mathbb E}.
\]
\end{corollary}

To approach question $(b)$, we need to introduce some more notation. For any $\xi\neq0\in\C$, we let $\Arg(\xi)\in(-\pi,\pi]$  be the principal branch of the argument and for $\xi=0$ we set $\Arg(0)=\pi.$ With this definition, $\Arg(\cdot)$ becomes a left continuous function on $\R$. Now, for any interval $[a,b]\subset\R$ we can define the angle in which this interval is seen at $\xi\in\C$ by
$$\ang(\xi,[a,b]):=|\Arg(a-\xi)-\Arg(b-\xi)|.$$
We define additively this angle for a system of disjoint closed intervals: if $\{[a_j,b_j]\}_{j=1}^m$ is such a system, then the angle in which it is seen at $\xi$ is defined by
\begin{equation}
\label{eq:agnleSys}
\angs(\xi):=\sum_{j=1}^m\ang(\xi,[a_j,b_j]).
\end{equation}
Note that $0\leq\angs(\xi)\leq\pi$ and $\angs(\xi)=\pi$ if and only if $\xi\in\cup[a_j,b_j]$. The notation $\theta(\xi)$ does not reflect the dependency on the system of intervals, but the latter will always be made clear. Further, for any point $z\in\C$ define the lower and upper characteristics $\overline m(z),\underline m(z)\in\Z_+$ as
\[\overline m(z) := \inf_U \overline m(z,U), \;\;\; \overline m(z,U) := \lim_{N\to\infty}\max_{n\geq N}\#\{S_n\cap U\},\]
and
\[\underline m(z) := \inf_U \underline m(z,U), \;\;\; \underline m(z,U) := \lim_{N\to\infty}\min_{n\geq N}\#\{S_n\cap U\},\]
respectively, where the infimum is taken over all open sets $U$ containing $z$ and $S_n$ is the set of poles of $g_n$ counting multiplicities. Clearly, $\underline m(z)\leq \overline m(z)$, $\underline m(z)=\infty$ if $z\in\spp$, and $\overline m(z)=0$ if  $z\notin\K$.

The forthcoming theorem implies that each pole of $F$ attracts at least as many poles of meromorphic approximants as its multiplicity and not much more. 
This is one answer to question $(b)$.

\begin{theorem}
\label{thm:tracingPoles}
Let $F$ and $\{g_n\}_{n\in\N}$  be as in Theorem \ref{thm:polesCriticalPoints}. Then
\begin{equation}
\label{eq:LowerM}
\underline m(\eta)\geq m(\eta), \;\;\; \eta\in\rspp,
\end{equation}
and
\begin{equation}
\label{eq:UpperM}
\sum_{\eta\in\rspp\setminus\spp}(\overline m(\eta)-m(\eta))(\pi-\angs(\eta)) \leq V,
\end{equation}
where
\begin{equation}
\label{eq:V}
V := V(\arm) + V_\W+(m+2s^\prime-1)\pi+2\sum_{\eta\in\rspp\setminus\spp}m(\eta)\angs(\eta),
\end{equation}
$s^\prime$ is the number of poles of $R$ on $\spp$ counting  multiplicities,
\begin{equation}
\label{eq:boundArgW}
V_\W:=\sup_{n\in\N}V(\Arg(w_n),S),
\end{equation}
and $\angs(\cdot)$ is the angle function for a system of $m$ intervals covering $\spp$.  
\end{theorem}

We shall prove in due course that indeed $V_\W<+\infty$.

Before we proceed, we shall derive several integral representations that are necessary for the proofs of the above-stated theorems. To alleviate notations, it is convenient to formally rewrite the
right-hand side of (\ref{eq:mainFun}) as a single Cauchy integral. 
For this, we introduce for $\eta=x_\eta+iy_\eta\in\C$ the distribution 
$\Phi_\eta=\chi(x-x_\eta)\otimes\delta(y-y_\eta)$, where
$\delta$ is the Dirac delta at $0$ and $\chi$ the characteristic function
of the non-negative semi-axis. For each $k\in\Z_+$ 
(the set of nonnegative integers), the partial derivative 
$\partial^{k+1}_x\Phi_\eta$
is an analytic functional (although $\Phi_\eta$ itself is not), acting on 
any function $h$ holomorphic in a neighborhood of $\eta$ according to the rule
$$\left<\partial^{k+1}_x\Phi_\eta,h\right>=h^{(k)}(\eta),$$
where $h^{(k)}$ indicates the $k$-th derivative. Therefore, if we define
$\Delta_\eta^{(k)}$ to be $\partial^{k}_x\Phi_\eta/k!$, we can formally
write
$$\int \frac{d\Delta_\eta^{(k)}(t)}{z-t}=\frac{1}{(z-\eta)^{k+1}},$$ 
and on rewriting $R_s(z)$ as
$$R_s(z)=\sum_{\eta\in\rspp}\sum_{k=0}^{m(\eta)-1}\frac{r_{\eta,k}}{(z-\eta)^{k+1}}, \;\;\; r_{\eta,k}\in\C,$$
we get 
$$R_s(z)=\int\frac{d\rmes(\xi)}{z-\xi},$$
where $\rmes$ is given by
\begin{equation}
\label{eq:deltaMeasure}
\rmes:=\sum_{\eta\in\rspp}\sum_{k=0}^{m(\eta)-1}r_{\eta,k}\Delta^{(k)}_{\eta}, \;\;\; \supp\left(\rmes\right)=\rspp.
\end{equation}
This way $F$ can be put in the form
$$F(z)=\int\frac{d\wmes(\xi)}{z-\xi}$$
with
\begin{equation}
\label{eq:newMes}
\wmes:=\mes+\rmes, \;\;\; ~~\supp\wmes=\wspp=\spp\cup\rspp,
\end{equation} 
which makes for a convenient notation.

Now, let $\{g_n\}$ be a sequence of irreducible critical points in $\pma(p)$ for some $p\in[2,\infty]$ (cf. (\ref{eq:criticalPoint})). Then (\cite[Prop. 9.1]{BS02})
\begin{equation}
\label{eq:decompHankel}
\begin{array}{ll}
A_F(v_n)(\xi)=\gamma_n^{1/2}\overline{\xi}\overline{\left(b_{n}j_nw_n^{p^\prime/2}\right)(\xi)}= \gamma_n^{1/2}\left(b_{n}j_nw_n^{p^\prime/2}\right)^\sigma(\xi), & p>2 \\
A_F(v_n)(\xi)=\overline{\xi}~\overline{(b_{n}u_n)(\xi)}=\left(b_{n}u_n\right)^\sigma(\xi), & p=2, 
\end{array}
\end{equation}
for a.e. $\xi\in\T$, where $j_n$ is some inner function, $u_n\in H^2$, and $h^\sigma(z):=z^{-1}\overline{h(1/\overline z)}$. Note that if $h\in H^2$ then $h^\sigma\in \bar H^2_0$ and vice versa. We remark that for the case $p=2$ equation (\ref{eq:decompHankel}) is an interpolation condition saying that $g_n$  
interpolates $F$ with order 2 at the reflection of its poles. 
The same can be said when $p\in(2,\infty]$, provided $g_n$ is analytic at the 
reflection of its poles, but this is no longer automatic because $g_n$ 
may no longer be rational.

As usual, we denote by $v_n$ an associated singular vector to $g_n$. According to (\ref{alternativemerodef}), 
each $g_n$ can be decomposed as $$g_n=b_n^{-1}\cdot h_n,$$ where $h_n\in H^p$. Moreover, we can write $b_n$ as ${q_n}/\widetilde q_n$, where $q_n\in\Mc_n$ and $\widetilde q_n$ is the reciprocal polynomial of $q_n$. Arguing like in \cite[Sec. 10]{BS02} (where $R_s$ is not present), equation (\ref{eq:decompHankel}) implies easily the following orthogonality relations 
\begin{equation}
\label{eq:orthRel}
\int t^kq_n(t)\frac{w_n(t)}{\widetilde{q}_n^2(t)}d\wmes(t)=0, ~~~ k=0,\ldots,n-1,
\end{equation}
where $w_n$ is the outer factor of $v_n$ and $\widetilde\mes$ is given by (\ref{eq:newMes}). Upon rewriting (\ref{eq:orthRel}) as
$$\int P_{n-1}(t)q_n(t)\frac{w_n(t)}{\widetilde{q}_n^2(t)}d\mes(t)+ \sum_{\eta\in\rspp}\sum_{k=0}^{m(\eta)-1}\frac{r_{\eta,k}}{k!} \left.\left(P_{n-1}(t)q_n(t)\frac{w_n(t)}{\widetilde{q}_n^2(t)}\right)^{(k)}\right|_{t=\eta}=0,$$
for all $P_{n-1}\in\Pc_{n-1}$ and taking $P_{n-1}$ to be a multiple of $Q_s$, 
these relations yield for $n>s$
\begin{equation}
\label{eq:newOrthRel}
\int t^kQ_s(t)q_n(t)\frac{w_n(t)}{\widetilde{q}_n^2(t)}d\mes(t)=0, ~~~ k=0,\ldots,n-s-1,
\end{equation}
where $Q_s$ was defined in (\ref{defQs}).

The following theorem is a result on the zero distribution of polynomials satisfying certain 
nonlinear orthogonality relations. As apparent from (\ref{eq:newOrthRel}), it will be the working tool of
our approach to the asymptotic behavior of poles of irreducible
critical points, in particular of poles of best or local best approximants.
We continue to denote by $I$ the convex hull of $\spp=\supp(\mes)$.
\begin{theorem}
\label{thm:weakLim}
Let $\{q_n\}_{n\in\N}$ be a sequence of polynomials of exact degree $n$ with all zeros in $\D$ satisfying the orthogonality relations
\begin{equation}
\label{eq:mainOrthRel}
\int t^kq_n(t)\frac{\omega_n(t)}{\widetilde{q}_n^2(t)}d\wmes(t)=0, \;\;\; k=0,\dots,n-1,
\end{equation}
where $\wmes=\mes+\rmes$ is given by (\ref{eq:mesDecomp}) and (\ref{eq:deltaMeasure})
with $S\subset(-1,1)$ and $S^\prime\subset\D$, while 
$\W=\{\omega_n\}_{n=1}^\infty$ is a family of complex measurable functions on 
the union of $\rspp$ and $I$, whose moduli are uniformly bounded 
above and below by positive constants, and whose arguments are 
smooth with uniformly bounded derivatives on $I$. 
Suppose further that $\mes\in\CoM$. 
Then the counting measures of the zeros of $q_n(z)=\prod_{j=1}^n(z-\xi_{j,n})$, 
namely $\nu_n:=(1/n)\sum_{j=1}^n\delta_{\xi_{j,n}}$, converge 
in the weak$^*$ sense to $\ged$.
\end{theorem}

The above-stated theorem is a generalization of Theorem 5.1 and 
Corollary 6.2 in \cite{BKT05}. The main difference here is that we add a 
distribution of the form (\ref{eq:deltaMeasure}) to the measure  
(\ref{eq:mainOrthRel}), i.e. $F$ may have polar singularities inside of the unit 
disk. As compared to the above reference, we simplified the proof somewhat
by using the two-constant theorem instead of weighted potential theory. 
To the author's knowledge these are the first published results about the zero 
distribution of polynomials satisfying nonlinear orthogonality equations like
(\ref{eq:mainOrthRel}) that are typical of rational or meromorphic approximation with free poles. In the Ph.D. thesis of R. K\"ustner \cite{thKus}, an analog of \cite[Theorem 5.1]{BKT05} is given when the measure $\mes$, instead of belonging to $\CoM$, has an argument of bounded variation and satisfies the 
so-called $\Lambda$-criterion introduced in \cite[Sec. 4.2]{StahlTotik}:
$$\cp\left(\left\{t\in\spp: \; \limsup_{r\to0}\frac{\Log(1/\mu[t-r,t+r])}{\Log(1/r)}<+\infty\right\}\right)=\cp(\spp).$$
Paralleling the arguments in \cite{thKus}, all the results in this section 
could be obtained under this weaker assumption, but the exposition 
would be heavier and we leave it to the interested reader to carry out the details.

\section{Proofs}
\label{sec:proofs}

The proof of Theorem \ref{thm:weakLim} relies on several auxiliary lemmas.

\begin{lemma}
\label{lem:aux}
With the previous notation the following statements hold true
\begin{itemize}
        \item [(a)] Let $\nu$ be a positive measure which has infinitely many points in its support and assume the latter is covered by finitely many disjoint intervals:
$\supp(\nu)\subseteq\cup_{j=1}^m[a_j,b_j]$. Let further $\psi$ be a function of 
bounded variation on $\supp(\nu).$ If for some integer $l$  we have
$$\int P_{l-1}(t)e^{i\psi(t)}d\nu(t)=0,~~~~~~\forall \,P_{l-1}\in\Pc_{l-1},$$
then
$$\sum_{j=1}^mV(\psi,[a_j,b_j])\geq(l-m+1)\pi.$$
        \item [(b)] Let $[a,b]\subset(-1,1)$ and $\xi\in\D.$ Define
\begin{equation}
\label{eq:defg}
g(\xi,t):=\Arg(t-\xi)-2\Arg(t-1/\bar\xi),
\end{equation}
where the term $2\Arg(t-1/\bar\xi)$ is omitted if $\xi=0$. Then
\begin{equation}
\label{eq:variationByAngle}
V(g(\xi,\cdot),[a,b])\leq\ang(\xi,[a,b]).
\end{equation}
        \item [(c)] Let $\psi$ be a real function of bounded variation on an
interval $[a,b]$, $\{a_n(x)\}$ a sequence of continuously differentiable
real functions with uniformly bounded derivatives on $[a,b]$, and $Q$ 
a polynomial. Then there exists a polynomial $T\neq0$ and a 
constant $\beta\in(0,\pi/32)$ such that
\begin{equation}
\label{eq:polT}
\left|\Arg\left(e^{i(\psi(x)+a_n(x))}Q(x)T(x)\right)\right|\leq\pi/2-2\beta
\end{equation}
for all $x\in[a,b]$ such that $T(x)Q(x)\neq0$ and all $n$ from some infinite sequence $\N_1\subset\N$.
        \item [(d)] Assume $I\subset(0,1)$ and $\{q_n\}$ is a sequence of 
polynomials of degree $m_n$ whose roots $\{\xi_{1,n},\ldots,\xi_{m_n,n}\}$ lie
in $\D$ and satisfy 
\[\sum_{j=1}^{m_n}(\pi-\ang(\xi_{j,n},I))\leq C\]
where the constant $C$ is independent of $n$. Then, to every $\epsilon>0$ there 
exists an integer $l$ such that, for all $n$ large enough, there is a 
polynomial $T_{l,n}$ of degree at most $l$ satisfying:
$$\left|\frac{\widetilde q_n(x)}{|\widetilde q_n(x)|}-T_{l,n}(x)\right|<\epsilon, \;\;\; x\in I.$$
In particular, the argument of $T_{l,n}(x)/\widetilde q_n(x)$ lies in the interval $(-2\epsilon,2\epsilon)$ when $n$ is large enough.
\end{itemize}
\end{lemma}
\begin{proof}$(a)$ This assertion follows from the proof of 
Lemma 3.2 in \cite{BKT05} if we put $d_n=l$ there.
\smallskip
\newline
$(b)$ When $\xi\notin I$, the proof of this statement is contained
in that of Lemma 5.2 in \cite{BKT05}. In the remaining cases, one can see 
by inspection that 
(\ref{eq:variationByAngle}) reduces to $0\leq\pi$ when $\xi=b$ and to 
$\pi\leq\pi$ when $\xi\in[a,b)$.
\smallskip
\newline
$(c)$ Observe that $\varphi(x)=\psi(x)+\Arg(Q(x))$ is a real function of 
bounded variation on $I$. Therefore by Lemma 3.4 in \cite{BKT05}, 
there exist a polynomial $T^*\neq0$ and a constant $\beta^*\in(0,\pi/16)$ such that
$$\left|\Arg\left(e^{i\psi(x)}Q(x)T^*(x)\right)\right|=
\left|\Arg\left(e^{i\varphi(x)}T^*(x)\right)\right|
\leq\pi/2-2\beta^*$$
for $x\in I$, $Q(x)T^*(x)\neq0$. For later use we also record that, by the very construction of $T^*$ in 
the cited lemma, its zeros belong to $I$ and are discontinuity points of $\varphi$. 

Let $K$ be such that $|a_n^\prime(x)|\leq K$ for all $n\in\N$ and $x\in
I$, where the superscript ``prime'' indicates the derivative.
By  Jackson's theorem (cf. {\it e.g.} \cite{Powell}) there is a constant $C>0$ 
and there are polynomials $\{T_{n,l}\}$ of degree at most $l$ such that 
$$\left|e^{-ia_n(x)}-T_{n,l}(x)\right|\leq\frac{CK}{l}.$$
Fix $l$ so large that $CK/l\leq\beta^*/3$. Being bounded of bounded degree,
the sequence $\{T_{n,l}\}$ has a subsequence converging 
uniformly on $I$ to a polynomial  $T_l$ of degree at most $l$. Therefore, for some subsequence $\N_1$ we obtain
$$\left|1-e^{ia_n(x)}T_l(x)\right|\leq\frac{\beta^*}{2}, \;\;\; n\in\N_1,$$
which implies that
$$\left|\Arg\left(e^{ia_n(x)}T_l(x)\right)\right|\leq\frac{\beta^*\pi}{4}<\beta^*, \;\;\; n\in\N_1.$$
Now inequality (\ref{eq:polT}) follows by taking $T=T^*T_l$, $\beta=\beta^*/2$, and using that $|\Arg(\xi_1+\xi_2)|\leq |\Arg(\xi_1)|+|\Arg(\xi_2)|$ for any $\xi_1,\xi_2\in\C$.
\smallskip
\newline
$(d)$ This is exactly what is proved in Lemma 5.4 of \cite{BKT05}.
\end{proof}
\begin{lemma}
\label{lem:anglBound}
Let $q_n(z)=\prod_{j=1}^n(z-\xi_{j,n})$ satisfy (\ref{eq:mainOrthRel}) 
for $|\xi_{j,n}|<1$ for $j=1,\ldots,n$, where
$\wmes=\mes+\rmes$ is given by (\ref{eq:mesDecomp}), (\ref{boundedvarphi}), and (\ref{eq:deltaMeasure}),
with $S\subset(-1,1)$ and $S^\prime\subset\D$, while 
$\omega_n$ is a complex-valued measurable function on $S\cup S^\prime$.
Consider a covering of $S$ by finitely many disjoint closed intervals:
 $S\subseteq I_m:=\bigcup_{j=1}^m[a_j,b_j]$. Then\footnote{Note that the hypothesis $\mes\in\CoM$ is not required for this lemma to hold.}
\begin{equation}
\label{eq:angleBound}
\sum_{j=1}^n(\pi-\angs(\xi_{j,n}))\leq V(\arm)+V({\rm arg}(\omega_n),S)+\sum_{\eta\in\rspp}m(\eta)\angs(\eta)+(m+s-1)\pi,
\end{equation}
where ${\rm arg}(\omega_n)$ is any argument function for $\omega_n$ on $S$ and
$m(\eta)$ is the multiplicity of $\eta$.
\end{lemma}
\begin{proof} If $\omega_n$ has no argument function of bounded variation on $S$, 
there is nothing to prove. Otherwise, we pick one and
extend it to the whole of $\R$ without increasing the variation,
as explained in Section \ref{sec:merom}. In particular, we get 
\[V({\rm arg}(\omega_n),S)=\sum_{j=1}^mV({\rm arg}(\omega_n),[a_j,b_j]).\]
As in the case of (\ref{eq:newOrthRel}), equation (\ref{eq:mainOrthRel}) yields
\begin{equation}
\label{eq:form2OR}
\int P_{n-s-1}(t)Q_s(t)q_n(t)\frac {\omega_n(t)}{\widetilde q_n^2(t)}d\mes(t)=0,
\end{equation}
where $P_{n-s-1}$ is any polynomial in $\Pc_{n-s-1}.$ 

Denote by $\psi_n(t)$ an argument function for $e^{i\arm(t)}Q_s(t)q_n(t)\omega_n(t)/\widetilde q_n^2(t),$ say
$$\psi_n(t)=\arm(t)+{\rm arg}(\omega_n(t))+\sum_{\eta\in\rspp}m(\eta)\Arg(t-\eta)+\sum_{i=1}^n\left(\Arg(t-\xi_{i,n})-2\Arg(t-1/\bar\xi_{i,n})\right),$$
where it is understood that $\Arg(t-1/\bar\xi_{i,n})$ is omitted when $\xi_{i,n}=0$. It is easy to see that $\psi_n$ is of bounded variation.  Then Lemma \ref{lem:aux}(a) with 
$$\psi=\psi_n, ~~~ d\nu(t)=\left|\frac{Q_s(t)q_n(t)\omega_n(t)}{\widetilde q_n^2(t)}\right|d|\mes|(t), ~~~\mbox{and}~~~ l=n-s$$ 
implies that
$$\sum_{j=1}^mV(\psi_n,[a_j,b_j])\geq(n-s-m+1)\pi.$$
So, we are left to show that
$$\sum_{j=1}^mV(\psi_n,[a_j,b_j])\leq V(\arm)+\sum_{j=1}^mV({\rm arg}(\omega_n),[a_j,b_j])+\sum_{\eta\in\rspp}m(\eta)\angs(\eta)+\sum_{i=1}^n\angs(\xi_{i,n}).$$
By the definition of $\psi_n$, we have
\begin{eqnarray}
\sum_{j=1}^mV(\psi_n,[a_j,b_j]) &\leq& \sum_{j=1}^mV(\arm,[a_j,b_j]) + \sum_{j=1}^mV({\rm arg}(\omega_n),[a_j,b_j]) \nonumber \\ 
{} &+& \sum_{j=1}^m\sum_{\eta\in\rspp}m(\eta)V(\Arg(\cdot-\eta),[a_j,b_j]) \nonumber \\
{} &+& \sum_{j=1}^m\sum_{i=1}^nV(g(\xi_{i,n},\cdot),[a_j,b_j]), \nonumber
\end{eqnarray}
where $g(\xi,t)$ was defined in (\ref{eq:defg}). The assertion of the lemma now follows from Lemma \ref{lem:aux}(b) and the fact that, by monotonicity, $V(\Arg(\cdot-\xi),[a,b])=\ang(\xi,[a,b])$.
\end{proof} 
\begin{corollary}
\label{cor:numZeros}
Let $\varphi$ and ${\rm arg}(\omega_n)$ have bounded variation on $S$, and assume further that $V({\rm arg}(\omega_n),S)<C$, where $C$ is independent on $n$.
Then, to each neighborhood $U$ of $\spp$, there exists a constant $k_U\in\N$ such that each $q_n$ has at most $k_U$ zeros outside of $U$ for $n$ large enough.
\end{corollary}
\begin{proof} Since $U$ is open, its intersection with $(-1,1)$ is a countable union of intervals. By compactness, a finite number of them will cover $\spp$, say $\cup_{j=1}^m(a_j,b_j)$. Apply Lemma \ref{lem:anglBound} to the closure of these intervals and observe that any zero of $q_n$ which lies outside of $U$ will contribute to the left-hand side of (\ref{eq:angleBound}) by more than some positive fixed constant which depends only on $U$. Since the right-hand side of (\ref{eq:angleBound}) does not depend on $n$ and is finite we can have only finitely many such zeros.
\end{proof}
\begin{proof}[Proof of Theorem \ref{thm:weakLim}] Hereafter we are going to use (\ref{eq:form2OR}) rather than (\ref{eq:mainOrthRel}) and we set $q_n(z)=\prod_{j=1}^n(z-\xi_{j,n})$.

We start by observing that we may suppose $\spp\subset(0,1)$. Indeed, if this is not the case, take  a negative number $w$ such that $-1<w<a$, where $[a,b]=I$ denotes the convex hull of $\spp$. Then $M_w(\spp)$, the image of $\spp$ under the M\"obius transformation $M_w(z):=(z-w)/(1-zw)$, is a subset of $(0,1)$. Moreover, the Green equilibrium measure is invariant under M\"{o}bius transformations, i.e., for any Borel set $E\subset M_w(\spp)$ we have that 
$$\mu_G^{M_w(\spp)}(E)=\mu_G^\spp(M_{-w}(E))$$
($M_{-w}$ is the inverse function of $M_w$). This implies that the weak$^*$ convergence of $\nu_n$ to $\ged$ is equivalent to that of $\nu_n^w$ to $\mu_G^{M_w(S)},$ where $\nu_n^w$ is the counting measure of the images of the zeros of $q_n$ under $M_w$. Now, if we let
\begin{eqnarray}
\ell_n(\tau)     &=& q_n(M_{-w}(\tau))(1+w\tau)^n, \nonumber \\
L_s(\tau)        &=& Q_s(M_{-w}(\tau))(1+w\tau)^s, \nonumber \\
p_{n-s-1}(\tau)  &=& P_{n-s-1}(M_{-w}(\tau))(1+w\tau)^{n-s-1}, \nonumber \\
\omega^*_n(\tau) &=& \omega_n(M_{-w}(\tau)), \nonumber
\end{eqnarray}
then $\ell_n$ is a polynomial of degree $n$ with zeros at $M_w(\xi_{j,n})$, $j=1,\ldots,n$. In addition, since $M_x(1/\bar\xi)=1/\overline{M_x(\xi)}$, $x\in(-1,1)$, we have that
$$\widetilde\ell_n(\tau)=\widetilde q_n(M_{-w}(\tau))(1+w\tau)^n.$$
Analogously, $L_s$ is a polynomial of degree $s$ with zeros at $M_w(\eta)$, $\eta\in\rspp$, and $p_{n-s-1}$ is an arbitrary polynomial of degree at most $n-s-1$. 
Let us show that $\ell_n$ satisfies orthogonality relations of type 
(\ref{eq:form2OR}) for a new measure, supported this time in $(0,1)$,
that still belongs to $\CoM$. 

With the above notation equation (\ref{eq:form2OR}) can be rewritten as
\begin{eqnarray}
0  &=& \int_\spp P_{n-s-1}(t)Q_s(t)q_n(t)\frac {\omega_n(t)}{\widetilde{q}_n^2(t)}e^{i\arm(t)}d|\mes|(t) \nonumber \\
{} &=& \int_{M_w(\spp)} (P_{n-s-1}Q_sq_n)(M_{-w}(\tau))\frac {\omega_n(M_{-w}(\tau))}{\widetilde{q}_n^2(M_{-w}(\tau))} e^{i\arm(M_{-w}(\tau))}d|\mes|(M_{-w}(\tau)) \nonumber \\
{} &=& \int_{M_w(\spp)}p_{n-s-1}(\tau)L_s(\tau)\ell_n(\tau)\frac {\omega^*_n(\tau)}{\widetilde{\ell}_n^2(\tau)}e^{i\arm^*(\tau)}d|\mes^*|(\tau), \nonumber
\end{eqnarray}
where $d|\mes^*|(\tau)=(1+w\tau)d|\mes|(M_{-w}(\tau))$ is a positive measure supported on $M_w(\spp)$, $\arm^*(\tau)=\arm(M_{-w}(\tau))$ is a function of bounded variation, and $\{\omega^*_n\}$ is a sequence of measurable functions whose moduli are uniformly bounded above and below, and whose arguments are smooth with uniformly bounded derivatives. Further, since Green functions are conformally invariant,$M_w(\spp)$ is regular so clearly $\mes^*\in\CoM$. This allows us to assume that $\spp\subset(0,1)$.

First we will suppose that {\it all zeros of the polynomials $q_n$ lie outside some fixed neighborhood of zero.}

For each $n$ denote by $\sigma_n$ the counting measure of zeros of $\widetilde q_n$. By the assumption that we just made there exists a compact set $K$ such 
that $0\notin K$ and $\supp(\nu_n)\subset K$. Then 
$\supp(\sigma_n)\subset \bar{K}^{-1}$ for all $n\in\N$, and
 $\bar{K}^{-1}$ is also compact. By Helly's selection theorem 
(cf. \cite[Thm. 0.1.3]{SaffTotik}) there exists a subsequence of natural 
numbers, $\N_1$, such that $\nu_n\cws\nu$ for $n\in\N_1$, where $\cws$ stands 
for weak$^*$ convergence. Denote by $\sigma$ the reflection of $\nu$ across 
the unit circle, i.e. $d\sigma(t)=d\nu\circ(1/\bar{t})$. It is easy to check that 
$\sigma_n\cws\sigma$. Observing that the assumptions on $\omega_n$ imply that 
the variation of its argument on $I$, thus {\it a fortiori} on $S$, is bounded independently of $n$, it follows from Corollary \ref{cor:numZeros} that
$\nu$ and $\sigma$ 
are probability measures such that $\supp(\nu)\subset\spp\subset (0,1)$ and 
$\supp(\sigma)\subset\bar{\spp}^{-1}=\spp^{-1}$.

{\it Claim:} it is enough to show that the logarithmic potential of $\nu-\sigma$ is a constant q.e. on $\spp$. Indeed, let $U^{\nu-\sigma}=D_1$ q.e. on $\spp$, where
$$U^{\nu-\sigma}=\int\log\frac{1}{|z-t|}d(\nu-\sigma)(t).$$
Then, since $U^\sigma$ is harmonic outside of $\spp^{-1}$, we have that $U^\nu$ is bounded quasi everywhere on $\spp$, hence everywhere by lower semi-continuity of potentials. Thus, $\nu$ has finite energy and by reflection so does $\sigma$. 
Moreover, for quasi every $z\in\spp^{-1}$, we have
\begin{eqnarray}
\label{eq:potConst}
U^{\nu-\sigma}(z) &=& \int\log\frac{1}{|z-t|}d(\nu-\sigma)(t)=\int\log\frac{1}{|z-1/x|}d\left[(\nu-\sigma)\circ(1/x)\right] \nonumber \\
{} &=& \int\log\left|\frac{x/z}{x-1/z}\right|d(\sigma-\nu)(x)= \int\log|x/z|d(\sigma-\nu)(x)-U^{\nu-\sigma}(1/z) \nonumber \\
{} &=& \int\log|x|d(\sigma-\nu)(x)-D_1 =:D_2, 
\end{eqnarray}
where we used that $\nu-\sigma$ has total mass zero. Now, denote by $\widehat\sigma$ the balayage of $\sigma$ onto $S$. Then
$$U^{\widehat\sigma}(t)=U^\sigma(t)+c(\sigma;\C\setminus\spp)$$ 
for quasi every $t$ on $\spp$. Thus, as $(\nu-\widehat\sigma)(\C)=0$ and  
since $\nu$ and $\widehat\sigma$ have finite energy, we get
\begin{eqnarray}
0  &=& \int D_1 \: d(\nu-\widehat\sigma)(z) = \int U^{\nu-\sigma}(z)d(\nu-\widehat\sigma)(z) = \int U^{\nu-\widehat\sigma}(z)d(\nu-\widehat\sigma)(z) \nonumber \\
{} &=& \int \log\frac{1}{|z-t|}d(\nu-\widehat\sigma)(t)d(\nu-\widehat\sigma)(z). \nonumber
\end{eqnarray}
But the energy of a signed measure is equal to zero if and only if the measure is zero (\cite[Lemma I.1.8]{SaffTotik}), provided that this measure is the difference of two positive measures with finite energy; thus, $\nu=\widehat\sigma$. Using (\ref{eq:potConst}), we can obtain in a similar fashion that $\sigma=\widehat\nu$, where $\widehat\nu$ is the balayage of $\nu$ onto $\spp^{-1}$. Hence, we proved that $\sigma-\nu$ is the equilibrium signed measure for the condenser $(\spp,\spp^{-1})$ (\cite[Thm. VIII.2.6]{SaffTotik}). Then Proposition A points (b) and (e) of the appendix ensures that $\nu=\ged$ is the Green equilibrium distribution relative to both $\D$ and $\C\setminus\spp^{-1}$. Since $\{\nu_n\}_{n\in\N_1}$ was an arbitrary weak$^*$ convergent subsequence, the whole sequence $\{\nu_n\}_{n\in\N}$ converges to $\mu_G^\spp$ in the weak$^*$ sense. {\it This proves the claim}.

Being left to prove that $U^{\nu-\sigma}$ is a constant q.e. on $\spp$,
suppose to the contrary that this is not true. Then there exist nonpolar Borel subsets of $\spp$, say $E_-$ and $E_+$, and two constants $d$ and $\tau>0$ such that
$$U^{\nu-\sigma}(x)\geq d+\tau \;\;\; \mbox{for} \;\;\; x\in E_+ \;\;\; \mbox{and} \;\;\; U^{\nu-\sigma}(x)\leq d-2\tau \;\;\; \mbox{for} \;\;\; x\in E_-.$$
In this case we claim that there exists $y_0\in\supp(\nu)$ such that
\begin{equation}
\label{eq:point}
U^{\nu-\sigma}(y_0)>d.
\end{equation}
Indeed, otherwise we would have that
\begin{equation}
\label{eq:ineqP}
U^\nu(x)\leq U^\sigma(x)+d, \;\;\; x\in\supp(\nu).
\end{equation}
Then the  principle of domination (\cite[Thm. II.3.2]{SaffTotik}) would yield that (\ref{eq:ineqP}) is true for all $z\in\C$, but this would contradict the existence of $E_+$.

Since all $\sigma_n$ are supported outside of the closed unit disk, the sequence of potentials $\{U^{\sigma_n}\}_{n\in\N_1}$ converges to $U^\sigma$ uniformly on compact subsets of $\D$. This implies that for any given sequence of points $\{y_n\}\subset\D$ such that $y_n\to y_0$ as $n\to\infty$, $n\in\N_1$, we have
\begin{equation}
\label{eq:sigmaConv}
\lim_{n\to\infty, \; n\in\N_1}U^{\sigma_n}(y_n)=U^\sigma(y_0).
\end{equation}
On the other hand all $\nu_n$, $n\in\N_1$, have their support in $\D$. So, by applying the principle of descent (\cite[Thm.I.6.8]{SaffTotik}) for the above sequence $\{y_n\}$, we obtain
\begin{equation}
\label{eq:nuConv}
\liminf_{n\to\infty, \; n\in\N_1} U^{\nu_n}(y_n)\geq U^\nu(y_0).
\end{equation}
Combining (\ref{eq:point}), (\ref{eq:sigmaConv}), and (\ref{eq:nuConv}) we get
\begin{equation}
\label{eq:nuSigmaConv}
\liminf_{n\to\infty \; n\in\N_1}U^{\nu_n-\sigma_n}(y_n)\geq U^{\nu-\sigma}(y_0)>d.
\end{equation}
Since $\{y_n\}$ was an arbitrary sequence in $\D$ converging to $y_0$, we deduce from (\ref{eq:nuSigmaConv}) that there exists $\rho>0$ such that, for any $y\in[y_0-2\rho,y_0+2\rho]$ and $n\in\N_1$ large enough, the following inequality holds
\begin{equation}
\label{eq:lowerBound}
U^{\nu_n-\sigma_n}(y)\geq d.
\end{equation}
Since 
$$U^{\nu_n-\sigma_n}(y)=\frac1n\log\left|\frac{1}{\alpha_n}\frac{\widetilde q_n(y)}{q_n(y)}\right|,$$
where $\alpha_n:=\prod_{j=1}^n|\xi_{j,n}|$, inequality (\ref{eq:lowerBound}) can be rewritten as
\begin{equation}
\label{eq:boundForBlaschke}
\left|\alpha_n\frac{q_n(y)}{\widetilde q_n(y)}\right|\leq e^{-nd}
\end{equation}
for any $y\in[y_0-2\rho,y_0+2\rho]$ and $n\in\N_1$ large enough.

Here we notice for later use that the above reasoning does not really require the polynomials $q_n$ to have exact degree $n$. Specifically, let $\{p_n\}$ be a sequence of 
monic polynomials of degree $d_n=n+o(n)$ where $o$ denotes the Landau symbol 
``little oh''. Moreover, suppose that the counting measures of 
their zeros normalized by $1/n$ rather than $1/d_n$ (so it may no longer be a probability measure) are supported on a fixed compact set of the complex plane. Call $\mu_n$ these measures and assume that they converge to $\nu$ in the weak$^*$ topology. 
In this case (\ref{eq:nuConv}) and (\ref{eq:nuSigmaConv}) still hold with $\nu_n$ replaced by $\mu_n$, at the cost perhaps 
of dropping finitely many terms of $\N_1$ and making $\rho$ smaller. Thus, we obtain that
\begin{equation}
\label{eq:upperBound}
\left|\alpha_n\frac{p_n(y)}{\widetilde q_n(y)}\right|\leq e^{-nd}
\end{equation}
for any $y\in[y_0-2\rho,y_0+2\rho]$ and $n\in\N_1$ large enough.

In another connection, since $U^{\nu-\sigma}(x)\leq d-2\tau$ on $E_-$, applying the lower envelope theorem (\cite[Thm. I.6.9]{SaffTotik}) gives us
\begin{equation}
\label{eq:lowerEnv}
\liminf_{n\to\infty, \; n\in\N_1}U^{\nu_n-\sigma_n}(x)=U^{\nu-\sigma}(x)\leq d-2\tau, \;\;\; \mbox{for q.e.} \;\;\; x\in E_-.
\end{equation}
Let $Z$ be a finite system of points from $(-1,1)$, to be specified later. Then by \cite{Ank83,Ank_CTP84} there exists\footnote{In \cite{BKT05} $S_0$ is not introduced, which makes there for a slightly incorrect argument in Theorem 3.1. 
An alternative remedy in that reference would be to apply the lower envelope 
theorem to $Tq_n$ rather than $q_n$, as they have the same asymptotic zero
distribution.} 
$\spp_0\subset\spp$ such that $\spp_0$ is regular$, \cp(E_-\cap\spp_0)>0$ and $\dist(Z,\spp_0)>0$, where $\dist(Z,\spp_0):=\min_{z\in Z}\dist(z,\spp_0)$. 
Put for simplicity $b_n(x)=q_n(x)/\widetilde q_n(x)$, 
which is a finite Blaschke product. 
Then 
$$U^{\nu_n-\sigma_n}(x)=-\frac1n\log|\alpha_nb_n(x)|$$
and by (\ref{eq:lowerEnv}), as $|\alpha_n|<1$, 
there exist $\N_2\subset\N_1$ and $x\in E_-\cap\spp_0$ such that 
$$|b_n(x)| \geq |\alpha_nb_n(x)|\geq e^{-n(d-\tau)}$$
for any $n\in\N_2$. Let $x_n$ be a point where $|b_n|$ reaches its maximum on $\spp_0$. Then
\begin{equation}
\label{eq:maxOfBlaschke}
M_n := \max_{x\in\spp_0}|b_n(x)| = |b_n(x_n)| \geq |\alpha_n| 
M_n = |\alpha_nb_n(x_n)| \geq e^{-n(d-\tau)}.
\end{equation}
Note that $M_n<1$ and therefore $d-\tau$ is necessarily positive. For simplicity, we shall denote $\D\setminus\spp_0$ by $D$. Since the modulus of a Blaschke product is bounded by 1 in the unit disk and $\log|b_n|$ is a subharmonic function, the two-constant theorem (\cite[Thm. 4.3.7]{Ransford}) yields the following pointwise estimate
\begin{equation}
\label{eq:logBlaschke}
\log|b_n(z)|\leq\omega_D(z,\spp_0)\log|M_n|
\end{equation}
for any $z\in D$, where $\omega_D(z,\spp_0)$ is the harmonic measure on $D$ (cf. \cite[Sec. 4.3]{Ransford}). Combining the last two inequalities we get
\begin{equation}
\label{eq:boundBlaschke}
|b_n(z)|\leq (M_n)^{\omega_D(z,\spp_0)}=M_n\left(\frac{1}{M_n}\right)^{1-\omega_D(z,\spp_0)}\leq M_ne^{n(d-\tau)(1-\omega_D(z,\spp_0))}
\end{equation}
for $z\in D$, and obviously also when $z\in\spp_0$, where 
$\omega_D(\cdot,\spp_0)=1$ for $S_0$ is regular. Moreover, 
by the regularity of $S_0$ again, it is known (\cite[Thm. 4.3.4]{Ransford}) 
that for any $x\in\spp_0$ $$\lim_{z\to x}\omega_D(z,\spp_0)=1$$
uniformly with respect to $x\in\spp_0$. 
Thus, for any $\delta>0$ there exists $r(\delta)<\dist(\spp_0,\T)$ such that for $z$ satisfying $\dist(z,\spp_0)\leq r(\delta)$ we have
$$1-\omega_D(z,S_0)\leq\frac{\delta}{d-\tau}.$$
This, together with (\ref{eq:boundBlaschke}), implies that for fixed $\delta$, to be adjusted later, we have
$$|b_n(z)|\leq M_ne^{n\delta}, \;\;\; |x_n-z|\leq r(\delta).$$
Note that $b_n$ is analytic in $\D$, which, in particular, yields 
for $|z-x_n|<r(\delta)$
$$b^\prime_n(z)=\frac{1}{2\pi i}\int_{|\xi-x_n|=r(\delta)}\frac{b_n(\xi)}{(\xi-z)^2}d\xi.$$ 
Thus, for any $z$ such that $|z-x_n|\leq r(\delta)/2$ we get
$$|b^\prime_n(z)|\leq\frac{1}{2\pi}\cdot\frac{4M_ne^{n\delta}}{r^2(\delta)}\cdot2\pi r(\delta)=\frac{4M_ne^{n\delta}}{r(\delta)}.$$
Now, for any $x$ such that 
\begin{equation}
\label{eq:circle}
|x-x_n|\leq\frac{r(\delta)}{8e^{n\delta}},
\end{equation}
the mean value theorem yields
$$|b_n(x)-b_n(x_n)|\leq\frac{4M_ne^{n\delta}}{r(\delta)}|x-x_n|=\frac{M_n}{2}.$$
Thus, for $x$ satisfying (\ref{eq:circle}), we have
$$|b_n(x)|\geq|b_n(x_n)|-|b_n(x)-b_n(x_n)|\geq M_n-\frac{M_n}{2}=\frac{M_n}{2}$$
and, by (\ref{eq:maxOfBlaschke}),
\begin{equation}
\label{eq:fromBelow}
|\alpha_nb_n(x)| \geq \frac{|\alpha_n| M_n}{2} \geq \frac12e^{-n(d-\tau)}.
\end{equation}
The estimates (\ref{eq:upperBound}), (\ref{eq:fromBelow}), together with the relation (\ref{eq:form2OR}) are the main ingredients in proving the claim that $U^{\nu-\sigma}$ is constant q.e. on $\spp$. To combine them we shall use a specific choice of $P_{n-s-1}$ in (\ref{eq:form2OR}). 

First, we pick a polynomial $T$ such that Lemma \ref{lem:aux}-(c) holds 
with $\psi=\arm$, $a_n=\Arg(\omega_n)$, $Q=Q_s$, and $[a,b]=I$ for 
$n\in\N_3\subset\N_2$. We denote by $k$ the degree of $T$. 
Second, for each $n\in\N_3$, we choose $T_{l,n}$ as in Lemma \ref{lem:aux}-(d) 
with $\epsilon=\delta/9$. Note that the use of Lemma \ref{lem:aux}-(d) 
is legitimate by
Lemma \ref{lem:anglBound} and our assumptions on $S$, $\varphi$ and $\omega_n$.
Since all $T_{l,k}$ are bounded on $I$ by definition and have degree at 
most $l$, which does not depend on $n$, there exists $\N_4\subset\N_3$ such 
that sequence $\{T_{l,n}\}_{n\in\N_4}$ converges uniformly to some polynomial 
$T_l$ on $I$. 
In particular, we have that $\deg(T_l)\leq l$ and the argument of 
$T_l(x)/\widetilde q_n(x)$ lies in $(-\delta/4,\delta/4)$ for $n\in\N_4$ 
large enough. Denote by $2\alpha$ the smallest even integer strictly 
greater than $2l+k+s$. As soon as $n$ is large enough, since 
$y_0\in\supp(\nu)$, there exist $\beta_{1,n},\ldots,\beta_{2\alpha,n}$, 
zeros of $q_n$, lying in 
$$O_\gamma\left([y_0-\rho,y_0+\rho]\right):=\left\{z\in\C: \; \dist\left(z,[y_0-\rho,y_0+\rho]\right)\leq \gamma\right\},$$
where $\gamma$, $0<\gamma<\rho$, will be specified later. Define 
$$P^*_n(z)=\frac{\overline{q_n(\overline z)}T(z)T^2_l(z)}{\prod_{j=1}^{2\alpha}(z-\overline \beta_{j,n})}.$$
The polynomial $P_n^*$ has degree $n-s-1$ or $n-s-2$, depending on the parity of $k+s$. 

Denote by $I_n\subset(0,1)$ the interval defined by (\ref{eq:circle}). By comparing (\ref{eq:boundForBlaschke}) with (\ref{eq:fromBelow}), it is clear that $I_n$ and $[y_0-2\rho,y_0+2\rho]$ are disjoint when $n\in\N_4$ is large enough.

Now, we choose $\gamma=\gamma(\rho)$ so small that
$$\left|\sum_{j=1}^{2\alpha}\Arg\left(\frac{1}{x-\overline\beta_{j,n}}\right)\right|=\left|\sum_{j=1}^{2\alpha}\Arg\left(x-\overline\beta_{j,n}\right)\right| \leq\delta/2, \;\;\; x\in \R\setminus[y_0-2\rho,y_0+2\rho].$$ 
Letting $\delta$ be such that $\delta<\beta$, the choices of $T$, $T_l$, 
and $P_n^*$ together imply
\begin{eqnarray}
{} && \left|\Arg\left(P^*_n(x)Q_s(x)q_n(x)\frac{\omega_n(x)}{\widetilde q_n^2(x)}e^{i\arm(x)}\right)\right| \nonumber \\  
{} &=& \left|\Arg\left(|q_n(x)|^2\cdot\prod_{j=1}^{2\alpha}\frac{1}{(x-\overline\beta_{j,n})}\cdot\frac{T_l^2(x)}{\widetilde q_n^2(x)}\cdot T(x)Q_s(x)\omega_n(x)e^{i\arm(x)}\right)\right| \nonumber \\
{} &\leq& \frac{\delta}{2}+\frac{\delta}{2}+\frac{\pi}{2}-2\beta\leq\pi/2-\delta, \nonumber
\end{eqnarray}
for $x\in I\setminus[y_0-2\rho,y_0+2\rho]$ except perhaps at points where 
$T$ or $Q_s$ are equal to zero. This means that for such $x$
\begin{eqnarray}
\re\left(|\alpha_n|^2(P^*_nQ_sq_n)(x)\frac{\omega_n(x)}{\widetilde q_n^2(x)}e^{i\arm(x)}\right) &\geq& \sin\delta\left|\alpha_n^2(P^*_nQ_sq_n)(x)\frac{\omega_n(x)}{\widetilde q_n^2(x)}e^{i\arm(x)}\right| \nonumber \\
\label{eq:bb1}
{} &\geq& \sin\delta\left|\frac{\alpha_n^2(b_n^2Q_sTT_l^2\omega_n)(x)}{\prod_{j=1}^{2\alpha}(x-\overline\beta_{j,n})}\right|\geq0.
\end{eqnarray}
Moreover, if $x\in\spp\setminus[y_0-2\rho,y_0+2\rho]$ satisfies (\ref{eq:circle}), then by (\ref{eq:maxOfBlaschke}) and (\ref{eq:fromBelow}) the above quantity is bounded from below by
\begin{eqnarray}
|T(x)Q_s(x)|\frac{\sin\delta~\min_{x\in[a,b]}|T_l(x)|^2~\inf_{n\in\N}\min_{x\in[a,b]}|\omega_n(x)|}{4(\diam(\spp)+2\rho)^{2\alpha}}e^{-2nd+2n\tau} \nonumber \\
 = c_1|T(x)Q_s(x)|e^{-2nd+2n\tau}, \nonumber
\end{eqnarray}
where $\diam(\spp):=\max_{x,y\in\spp}|x-y|$ and
$$c_1:=\frac{\sin\delta~\min_{x\in[a,b]}|T_l(x)|^2~\inf_{n\in\N}\min_{x\in[a,b]}|\omega_n(x)|}{4(\diam(\spp)+2\rho)^{2\alpha}}>0$$
by the construction of $T_l$ and the uniform boundedness of $\{|\omega_n|\}$ from below. Thus
\begin{eqnarray}
\label{eq:contrBelow}
{} && \re\left(\int_{\spp\setminus[y_0-2\rho,y_0+2\rho]}|\alpha_n|^2P^*_n(t)Q_s(t)q_n(t)\frac{\omega_n(t)}{\widetilde q_n^2(t)}e^{i\arm(t)}d|\mes|(t)\right) \nonumber \\
{} &\geq& \sin\delta\int_{\spp\setminus[y_0-2\rho,y_0+2\rho]}\left|\alpha_n^2P^*_n(t)Q_s(t)q_n(t)\frac{\omega_n(t)}{\widetilde q_n^2(t)}e^{i\arm(t)}\right|d|\mes|(t) \nonumber \\
{} &\geq&  c_1e^{-2nd+2n\tau}\int_{\spp\cap I_n}|T(t)Q_s(t)|d|\mes|(t)\geq c_2e^{-2nd+n(2\tau-L\delta)}.
\end{eqnarray}
The last inequality is true by the following argument. First observe that from (\ref{eq:fromBelow}) and \ref{eq:boundForBlaschke}) that $I_n\cap[y_0-2\rho,y_0+2\rho]=\emptyset$ for all $n$ large enough. Next, recall that $x_n$, the middle point of $I_n$, belongs to $\spp_0$, where $\dist(\spp_0,Z)>0$ and $Z$ is a finite system of points that we now choose to be the zeros of $TQ_s$ in $(-1,1)$ if any. Then $TQ_s$, which is independent of $n$, is uniformly bounded below on $I_n$ and (\ref{eq:contrBelow}) follows from this and hypothesis $\CoM$ point (2). On the other hand (\ref{eq:boundForBlaschke}), and (\ref{eq:upperBound}) applied with $p_n=P_n^*/$(leading coefficient of $TT_l^2$), yield that
\begin{equation}
\label{eq:contrAbove}
\left|\int_{[y_0-2\rho,y_0+2\rho]}|\alpha_n|^2P^*_n(t)Q_s(t)q_n(t)\frac{\omega_n(t)}{\widetilde q_n^2(t)}e^{i\arm(t)}d|\mes|(t)\right|\leq c_3e^{-2nd},
\end{equation}
where we used uniform boundedness of $\{|\omega_n|\}$ from above. 
Choosing $\delta$ so small that $2\tau-L\delta>0$, which is possible, the bound in (\ref{eq:contrBelow}) becomes bigger than the bound in (\ref{eq:contrAbove}) for $n$ large enough. But this is impossible, since by (\ref{eq:form2OR}) the sum of these two integrals must be zero.

We just completed the case when all the zeros of polynomials $q_n$ stay away from the point zero. Now we shall consider the general situation. Let $\epsilon>0$ be such that $U:=\D\setminus\{z:\; |z|\leq\epsilon\}$ is a neighborhood of $\spp$. Corollary \ref{cor:numZeros} says that there exists a constant $k_U\in\N$ such that each $q_n$ has at most $k_U$ zeros outside of $U$, that is zeros which have modulus less than $\epsilon$. In this case,
from any sequence of natural numbers, we can extract a subsequence, 
say $\N_0$, such that for some number $m\leq k_U$, $q_n$ has 
exactly $m$ zeros outside of $U$ 
for each $n\in\N_0$. Denote these zeros $\xi_{1,n},\ldots,\xi_{m,n}$, 
and consider the polynomials
$$q_n^*(z):=\frac{\widetilde q_n(z)}{\prod_{j=1}^m(1-z\overline \xi_{j,n})}, \;\;\; n\in\N_0.$$
Then the sequence $\{q_n\}_{n\in\N_0}$ will satisfy the following weighted orthogonality relation:
$$\int P_{n-s-1}(t)Q_s(t)q_n(t)\frac{\omega^*_n(t)}{(q_n^*(t))^2}d\mes(t)=0, \;\;\; P_{n-s-1}\in\Pc_{n-s-1},$$
where 
$$\omega^*_n(t):=\omega_n(t)\prod_{j=1}^m(1-t\overline\xi_{j,n})^{-2}, \;\;\; t\in I.$$

In what follows we are going to stress the modifications needed to adapt the previous proof to the present case. Let, as before, $\N_1\subset\N_0$ be a subsequence of natural numbers such that $\nu_n\cws\nu$, $n\in\N_1$. Because we only discarded a fixed number of zeros from $\widetilde q_n(z)$
to obtain $q_n^*$, the counting measures $\sigma_n^*$ of the zeros of 
$q_n^*$ (normalized by $1/n$), again converge weak$^*$ to $\sigma$. Since $U^{\sigma_n^*}$ enjoys all the relevant properties of $U^{\sigma_n}$, inequalities (\ref{eq:boundForBlaschke}) and (\ref{eq:upperBound}) remain valid with $\widetilde q_n$ replaced by $q_n^*$. 

Further, define $b_n(x)$ as $q_n(x)/q_n^*(x)$. In this case, $b_n$ is no 
longer a Blaschke product, but rather a Blaschke product multiplied 
by the polynomial$\prod_{j=1}^m(z-\xi_{j,n})$. Then we 
get instead of (\ref{eq:logBlaschke}) that
$$\log|b_n(z)|\leq\omega_D(z,S_0)\log|M_n|+(1-\omega_D(z,S_0))m\log2,$$
 and (\ref{eq:boundBlaschke}) can be replaced by
$$|b_n(z)|\leq M_ne^{n(d-\tau+m\log2/n)(1-\omega_D(z,\spp_0))}\leq M_ne^{n(d-\tau+1)(1-\omega_D(z,\spp_0))},$$
for $n$ large enough. This yields an insignificant modification of $r(\delta)$ (we should make $1-\omega_D(z,\spp_0)$ less than $\delta/(d-\tau+1)$ rather than just $\delta/(d-\tau)$).
Lemma \ref{lem:aux}(d) can be applied to the polynomials $q_n^*$ rather than $\widetilde q_n$ without change.

Therefore we are left to show that $\{\omega_n^*\}$ is uniformly bounded above and below on $I$, and that its arguments are smooth with uniformly bounded 
derivatives on $I$ with respect to $n$. The uniform boundedness of 
$\{\omega_n^*\}$ easily follows from the estimates 
$$\left(\frac12\right)^{2m}\leq\prod_{j=1}^m\left|\frac{1}{(1-t\overline\xi_{j,n})}\right|^2\leq\left(\frac{1}{1-\epsilon}\right)^{2m}.$$

Since none of the $\xi_{j,n}$, $j=1,\ldots,m$ can come close to
$I^{-1}$, $\Arg(1-t\overline\xi_{j,n})$ is a smooth function on $I$ whose 
derivative ${\rm Im}\xi_{j,n}/(1-t\overline\xi_{j,n})$ is uniformly bounded 
there independently of $j$ and $n$. 
Then the rest of the assumptions on $\{\omega_n^*\}$ follows from 
the representation
$${\rm arg}(\omega_n^*(t))={\rm arg}(\omega_n(t))-2\sum_{j=1}^m\Arg(1-t\overline\xi_{j,n}).$$
This completes the proof of the theorem.
\end{proof}

The forthcoming lemma is needed for the proof of Theorem \ref{thm:polesCriticalPoints}. Recall that $\mathcal{H}$, a family of functions analytic in some fixed domain of the complex plane, is called {\it normal} if each sequence of functions from $\mathcal{H}$ contains a locally uniformly convergent subsequence.

\begin{lemma}
\label{lem:normalWeight}
Let $p\in(2,\infty]$ and $\{g_n\}_{n\in\N}$ be a sequence of a sequence of 
irreducible critical points of order $n$ of $\pma(p)$ for $F$ given by (\ref{eq:mainFun}) and (\ref{defQs}) with 
$\mes$ satisfying\footnote{Note that we do not require the hypothesis $\mes\in\CoM$ to hold. It is sufficient for the lemma to hold to have a measure with an argument of bounded variation and infinitely many points in the support.} (\ref{eq:mesDecomp}) and (\ref{boundedvarphi}). 
Further, let $v_n$ be an associated singular vector to $g_n$ with inner-outer factorization given by $v_n=b_n\cdot w_n$ for some Blaschke product $b_n$ 
and $w_n$ an outer function in $H^{p^\prime}$. Then the families $\W:=\{w_n\}$ and $\W_{p^\prime}:=\left\{w_n^{p^\prime/2}\right\}$ are normal in $\D$ and $\overline\C\setminus\wspp^*$ respectively, where $\wspp^*$ denotes the reflection of $\wspp$ across $\T$. Moreover, any limit point of $\W$ is zero free in $\D$. 
\end{lemma}
\begin{proof} The main idea of the proof was given in \cite[Thm. 10.1]{BS02}. The necessary modification for the case of complex measures with argument of bounded variation were given in \cite[Prop. 6.3]{BMSW06}. Only a simple adjustment is needed in the present case where the approximated function may have polar singularities inside $\D$. Namely, after the initial choice of a polynomial $T$ in \cite[Prop. 6.3, Eq. 6.12]{BMSW06} has been made, one systematically applies the arguments in this proposition with $TQ^2$ instead of $T$. Another modification that one has to introduce is to define $G_n$ in \cite[Thm. 10.1]{BS02} as
$$G_n(z) := \overline{(b_nQw_n^{1/2})(\bar z)},$$
in order to obtain the desired bound in \cite[Eq. 6.13]{BMSW06}. The interested reader can find the full proof of this lemma in \cite[Lem. 3.11]{thYat}. 
\end{proof}

Note that in the previous lemma the normality of $\mathcal{W}$ in $\D$ was clear beforehand by the Cauchy formula, since $\|w_n\|_{p'}=1$, but the nonzeroing of every limit point was not.

\begin{proof}[Proof of Theorem \ref{thm:polesCriticalPoints}] Let $v_n$ be a singular vector associated to $g_n$ with inner-outer factorization given by $v_n=b_n\cdot w_n$ for some Blaschke product $b_n=q_n/\widetilde q_n$ and some outer function $w_n\in H^{p^\prime}$, $\|w_n\|_{p^\prime}=1$, where $w_n\equiv1$ when $p=2$. The poles of $g_n$ are exactly the zeros of $q_n$. Moreover, $\{q_n\}$ is a sequence of polynomials satisfying weighted non-Hermitian orthogonality relations (\ref{eq:orthRel}). Thus, the assertion of the theorem will follow from Theorem \ref{thm:weakLim} if only $\W=\{w_n\}$ is uniformly bounded above and below on $I$, the convex hull of $\spp$, and if it is a family of functions whose arguments are smooth with uniformly bounded derivatives on $I$. In the case $p=2$ this is trivial since each $w_n\equiv1$. In the case $p\in(2,\infty]$ Lemma \ref{lem:normalWeight} says that $\W$ is a normal family. Thus, it is uniformly bounded above on $I$. Moreover, since all limit points of $\W$ are zero free in $\D$, this family is uniformly bounded below on $I$ (in fact on any compact subset of $\D$). Further, the derivatives again form a normal family and so does the logarithmic derivative $w_n^\prime/w_n$ in $\D$. Since the imaginary part of the latter is equal to $d~{\rm arg}(w_n)/dt$ on $I$, we see that the desired conditions on $\W$ are satisfied.
\end{proof}

Before we prove Theorem \ref{thm:convergenceCapacity} we shall need one auxiliary lemma.
\begin{lemma}
\label{lem:TwoConstant}
Let $D$ be a domain in $\overline\C$ with non-polar boundary, $K^\prime$ be a compact set in $D$, and $\{u_n\}$ be a sequence of subharmonic functions in $D$ such that
\[
u_n(z) \leq M -\epsilon_n, \;\;\; z\in D,
\]
for some constant $M$ and a sequence $\{\epsilon_n\}$ of positive numbers decaying to zero. Further, assume that there exist a compact set $K^\prime$ and positive constants $\epsilon^\prime$ and $\delta^\prime$, independent of $n$, for which holds
\[
u_n(z) \leq M - \epsilon^\prime, \;\;\; z\in K_n\subset K^\prime, \;\;\; \cp(K_n)\geq\delta^\prime.
\]
Then for any compact set $K\subset D\setminus K^\prime$ there exists a positive constant $\epsilon(K)$ such that
\[
u_n(z) \leq M - \epsilon(K), \;\;\; z\in K,
\]
for all $n$ large enough. 
\end{lemma}
\begin{proof} Let $\omega_n$ be the harmonic measure for $D_n := D\setminus K_n$. Then the two-constant theorem \cite[Thm. 4.3.7]{Ransford} yields that
\begin{eqnarray}
u_n(z) &\leq& (M - \epsilon^\prime)\omega_n(z,K_n) + (M-\epsilon_n)(1-\omega_n(z,K_n)) \nonumber \\
{}     &\leq& M - (\epsilon^\prime-\epsilon_n)\omega_n(z,K_n), \;\;\; z\in D_n. \nonumber
\end{eqnarray}
Thus, we need to show that for any $K\subset D\setminus K^\prime$ there exists a constant $\delta(K)>0$ such that
\[
\omega_n(z,K_n) \geq \delta(K), \;\;\; z\in K.
\]
Assume to the contrary that there exists a sequence of points $\{z_n\}_{n\in\N_1}\subset K$, $\N_1\subset\N$, such that
\begin{equation}
\label{eq:contrary}
\omega_n(z_n,K_n) \to 0 \;\; \mbox{as} \;\; n\to\infty, \;\; n\in\N_1.
\end{equation}
By \cite[Theorem 4.3.4]{Ransford}, $\omega_n(\cdot,K_n)$ is the unique bounded harmonic function in $D_n$ such that
\[
\lim_{z\to\zeta} \omega(z,K_n) = 1_{K_n}(\zeta)
\]
for any regular $\zeta\in\partial D_n$, where $1_{K_n}$ is the characteristic function of $\partial K_n$. Then it follows from (\ref{eq:GreenEqual}) of the appendix that
\begin{equation}
\label{eq:HarmEqPot}
\cp(K_n,\partial D) U_D^{\mu_{(K_n,\partial D)}} \equiv \omega_n(\cdot,K_n),
\end{equation}
where $\mu_{(K_n,\partial D)}$ is the Green equilibrium measure on $K_n$ relative to $D$. Since all the measures $\mu_{(K_n,\partial D)}$ are supported in the compact set $K^\prime$, there exists a probability measure $\mu$ such that
\[
\mu_{(K_n,\partial D)} \cws \mu \;\; \mbox{as} \;\; n\to\infty, \;\; n\in\N_2\subset\N_1.
\]
Without loss of generality we may suppose that $z_n\to z^*\in K$ as $n\to\infty$, $n\in\N_2$. Let, as usual, $g_D(\cdot,t)$ be the Green function for $D$ with pole at $t\in D$. Then, by the uniform equicontinuity of $\{g_D(\cdot,t)\}_{t\in K^\prime}$ on $K$, we get
\[
U_D^{\mu_{(K_n,\partial D)}}(z_n) \to U_D^{\mu}(z^*) \neq 0 \;\; \mbox{as} \;\; n\to\infty, \;\; n\in\N_2.
\]
Therefore, (\ref{eq:contrary}) and (\ref{eq:HarmEqPot}) necessarily mean that
\begin{equation}
\label{eq:contrary2}
\cp(K_n,\partial D) \to 0 \;\; \mbox{as} \;\; n\to\infty, \;\; n\in\N_2.
\end{equation}
By definition, $1/\cp(K_n,\partial D)$ is the minimum among Green energies of probability measures supported on $K_n$. Thus, the sequence of Green energies of the logarithmic equilibrium measures on $K_n$, $\mu_{K_n}$, diverges to infinity by (\ref{eq:contrary2}). Moreover, since 
\[
\left\{g(\cdot,t)+\log|\cdot-t|\right\}_{t\in K^\prime}
\]
is a family of harmonic functions in $D$ whose moduli are uniformly bounded above on $K^\prime$, the logarithmic energies of $\mu_{K_n}$ diverge to infinity. In other words,
\[
\cp(K_n) \to 0 \;\; \mbox{as} \;\; n\to\infty, \;\; n\in\N_2,
\]
which is impossible by the initial assumptions. This proves the lemma.
\end{proof}

\begin{proof}[Proof of Theorem \ref{thm:convergenceCapacity}] To prove the convergence in capacity we first establish an integral representation for the error $(F-g_n)$. As usual we denote by $v_n=b_nw_n$ a singular vector associated to $g_n$, where $b_n$ is a Blaschke product of degree $n$ and $w_n$ is an outer function. By (\ref{eq:DefCrPt}) and Fubini-Tonelli's theorem, we have for $z\in\D\setminus\wspp$
\begin{eqnarray}
\label{eq:errorF1}
r_n(z) &:=& (F-g_n)(z)=\frac{A_F(v_n)(z)}{v_n(z)} = \frac{\Pj_-(Fv_n)(z)}{v_n(z)} \nonumber \\
{} &=& \frac{1}{v_n(z)}\int_\T\frac{(Fv_n)(\xi)}{z-\xi}\frac{d\xi}{2\pi i} = \frac{1}{v_n(z)}\int_\T\int\frac{v_n(\xi)}{(z-\xi)(\xi-t)}d\wmes(t)\frac{d\xi}{2\pi i}\nonumber \\
{} &=& \frac{1}{v_n(z)}\int\frac{v_n(t)}{z-t}d\wmes(t) = \frac{\widetilde q_n(z)}{q_n(z)w_n(z)}\int\frac{q_n(t)}{z-t}\,\frac{w_n(t)}{\widetilde q_n(t)}d\wmes(t).
\end{eqnarray}
In another connection, the orthogonality relations (\ref{eq:orthRel}) yield
$$\int\frac{\widetilde q_n(z)-\widetilde q_n(t)}{z-t}~q_n(t)~\frac{w_n(t)}{\widetilde q_n^2(t)}d\wmes(t)=0.$$
This, in turn, implies that (\ref{eq:errorF1}) can be rewritten as
\begin{equation}
\label{eq:intermidForm}
r_n(z)=\frac{\widetilde q_n^2(z)}{q_n(z)w_n(z)}\int\frac{q_n(t)}{z-t}\frac{w_n(t)}{\widetilde q_n^2(t)}d\wmes(t), \;\;\; z\in\D\setminus\wspp.
\end{equation}
Since the majority of the zeros of $q_n$ approach $\spp$ by Corollary \ref{cor:numZeros}, we always can choose $s$ of them, say $\xi_{1,n},\ldots,\xi_{s,n}$, in such a manner that the absolute values of
\begin{equation}
\label{eq:Def1}
l_{s,n}(z):=\prod_{j=1}^s(z-\xi_{j,n}) \;\;\; \mbox{and} \;\;\; \widetilde l_{s,n}(z):= z^s\overline{l_{s,n}(1/\bar z)}
\end{equation}
are uniformly bounded above and below on compact subsets of $\C\setminus\spp$ and in $\overline\D$, respectively, for all $n$ large enough. Using orthogonality relations (\ref{eq:orthRel}) once again and since $Q_s$ vanishes at each $\eta\in\rspp$ with multiplicity $m(\eta)$, we can rewrite (\ref{eq:intermidForm}) as
\begin{equation}
\label{errmanip}
r_n(z) = \frac{\widetilde q^2_n(z)}{(q_nq_n^*Q_sw_n)(z)} \int\frac{(q_nq_n^*Q_sw_n)(t)}{\widetilde q_n^2(t)} \frac{d\mes(t)}{z-t},
\end{equation}
where $q_n^*(z) := \overline{q_n(\bar z)}/\overline{l_{s,n}(\bar z)}$. Set
\begin{equation}
\label{eq:Def2}
B_n(z) := \int \frac{(|q_n^*|^2Q_sl_{s,n}w_n)(t)}{\widetilde q_n^2(t)}\frac{d\mes(t)}{z-t}, \;\;\; z\in\C\setminus\spp,
\end{equation}
so that
\[
r_n(z) = \frac{\widetilde q^2_n(z)B_n(z)}{(q_nq_n^*Q_sw_n)(z)}.
\]
First, we show that 
\begin{equation}
\label{eq:ConvIn}
|B_n|^{1/2n}\cic\exp\{-1/\cp(\spp,\T)\}
\end{equation}
on compact subsets of $\C\setminus\spp$. Denote, as usual,  $b_n=q_n/\widetilde q_n$. Then for any compact set $K\subset\C\setminus\spp$ there exists a constant $c(K)$, independent of $n$, such that
\begin{equation}
\label{eq:UpperBoundIn}
|B_n(z)| \leq c(K)\|b_n\|_\spp^2, \;\;\; z\in K,
\end{equation}
by the choice of $l_{s,n}$ and Lemma \ref{lem:normalWeight}.  Let $\nu_n$ be the counting measures of zeros of $b_n$. Then
\begin{eqnarray}
\limsup_{n\to\infty}\left|b_n(t)\right|^{1/n} &=& \limsup \exp\left\{-U^{\nu_n}_\D(t)\right\} = \exp\left\{- \liminf U^{\nu_n}_\D(t)\right\} \nonumber \\
\label{eq:LET}
{} &=& \exp\left\{-U^{\ged}_\D(t)\right\} = \exp\{-1/\cp(\spp,\T))\} \;\; \mbox{for q.e.} \;\; t\in\spp
\end{eqnarray}
by Theorem \ref{thm:polesCriticalPoints}, the lower envelope theorem \cite[Thm. I.6.9]{SaffTotik}, and (\ref{eq:GreenEqual}) of the appendix. Moreover, by the principle of descent \cite[Thm. I.6.8]{SaffTotik}, we get that
\begin{equation}
\label{eq:POD}
\limsup_{n\to\infty}\left|b_n(t)\right|^{1/n} \leq \exp\{-1/\cp(\spp,\T)\}
\end{equation}
uniformly on $\spp$. It is immediate from (\ref{eq:LET}) and (\ref{eq:POD}) that, in fact,
\begin{equation}
\label{eq:LimitNorms}
\lim_{n\to\infty}\left\|b_n\right\|_\spp^{1/n} = \exp\{-1/\cp(\spp,\T)\}.
\end{equation}
Suppose now that (\ref{eq:ConvIn}) is false. Then there would exist a compact set $K^\prime\subset\C\setminus\spp$ and $\epsilon^\prime>0$ such that
\begin{equation}
\label{eq:ContrIn}
\cp\left\{z\in K^\prime:~\left||B_n(z)|^{1/2n}-\exp\{-1/\cp(\spp,\T)\}\right|\geq\epsilon^\prime\right\} \not\to 0.
\end{equation}
Combining (\ref{eq:ContrIn}), (\ref{eq:LimitNorms}), and (\ref{eq:UpperBoundIn}) we see that there would exist a sequence of compact sets $K_n\subset K^\prime$, $\cp(K_n)\geq\delta^\prime>0$, such that
\begin{equation}
\label{eq:FalseUpperBoundIn}
|B_n(z)|^{1/2n} \leq \exp\{-1/\cp(\spp,\T)\} - \epsilon^\prime, \;\;\; z\in K_n,
\end{equation}
for all $n$ large enough. Now, let $\Gamma$ be a closed Jordan curve that separates $\spp$ from $K^\prime$ and contains $\spp$ in the bounded component of its complement. Observe that $(1/2n)\log|B_n|$ is a subharmonic function in $\C\setminus\spp$. Then (\ref{eq:UpperBoundIn}), (\ref{eq:LimitNorms}), and (\ref{eq:FalseUpperBoundIn}) enable us to apply Lemma \ref{lem:TwoConstant} with $M=-1/\cp(\spp,\T)$ which yields that there exists $\epsilon(\Gamma)>0$ such that
\begin{equation}
\label{eq:FalseOnGamma}
|B_n(z)|^{1/2n} \leq \exp\{-1/\cp(\spp,\T)-\epsilon(\Gamma)\}
\end{equation}
uniformly on $\Gamma$ and for all $n$ large enough. Define
\[
J_n : = \left|\int_\Gamma T_l^2(z)T(z)\overline{l_{s,n}(\bar z)}B_n(z)\frac{dz}{2\pi i}\right|,
\]
where the polynomials $T_l$ and $T$ are chosen as in Theorem \ref{thm:polesCriticalPoints} (see discussion after (\ref{eq:fromBelow})). Then if the limit in (\ref{eq:ConvIn}) would not hold, we would get (\ref{eq:FalseOnGamma}) and subsequently
\begin{equation}
\label{eq:ContrPart1}
\limsup_{n\to\infty} J_n^{1/2n} \leq \exp\{-1/\cp(\spp,\T)-\epsilon(\Gamma)\}.
\end{equation}
In another connection, the Cauchy integral formula yields that
\begin{eqnarray}
J_n &=& \left|\int_\Gamma T_l^2(z)T(z)\overline{l_{s,n}(\bar z)}\left(\int\frac{(|q_n^*|^2Q_sl_{s,n}w_n)(t)}{\widetilde q_n^2(t)}\frac{d\mes(t)}{z-t}\right) \frac{dz}{2\pi i}\right| \nonumber \\
\label{eq:BelowOnGamma1}
{} &=& \left|\int|q_n^2(t)|\frac{T_l^2(t)}{\widetilde q_n^2(t)}(TQ_sw_n)(t)e^{i\arm(t)}d|\mes|(t)\right|. 
\end{eqnarray}
Exactly as in (\ref{eq:bb1}), we can write
\begin{equation}
\label{eq:BelowOnGamma2}
\re\left(\frac{(T_l^2TQ_sw_n)(t)e^{i\arm(t)}}{\widetilde q_n^2(t)}\right) \geq \sin(\delta)\left|\frac{(T_lTQ_sw_n)(t)}{\widetilde q_n^2(t)}\right|, \;\;\; t\in I,
\end{equation}
where $I$ is the convex hull of $\spp$ and $\delta>0$ has the same meaning as in Theorem \ref{thm:polesCriticalPoints} (see construction after (\ref{eq:boundBlaschke})). Thus, we derive from (\ref{eq:BelowOnGamma1}) and (\ref{eq:BelowOnGamma2}) that
\begin{equation}
\label{eq:BelowOnGamma3}
J_n \geq \sin(\delta) \int |b_n^2(t)| ~ |(T_l^2TQ_sw_n)(t)| d|\mes|(t).
\end{equation}
Let $\spp_0$ be a closed subset of $\spp$ of positive capacity that lies at positive distance from the zeros of $TQ_s$ on $I$ (see Theorem \ref{thm:polesCriticalPoints} for the existence of this set). Further, let  $x_n\in\spp_0$ be such that
\[
\|b_n\|_{\spp_0} = |b_n(x_n)|.
\]
It follows from (\ref{eq:LET}) and (\ref{eq:LimitNorms}) that
\[
\lim_{n\to\infty}\|b_n\|_{\spp_0}^{1/n} = \exp\{-1/\cp(\spp,\T)\},
\]
and therefore
\[
\|b_n\|_{\spp_0} \geq \exp\{-n(\epsilon+1/\cp(\spp,\T))\}
\]
for any $\epsilon>0$ and all $n$ large enough. Proceeding as in Theorem \ref{thm:polesCriticalPoints} (see equations (\ref{eq:circle}) and (\ref{eq:fromBelow})), we get that
\begin{equation}
\label{eq:BelowOnGamma4}
|b_n(t)| \geq \frac12\exp\{-n(\epsilon+1/\cp(\spp,\T))\}, \;\;\; t\in I_n,
\end{equation}
where
\[
I_n:=\left\{x\in\spp_0:~|x-x_n|\leq r_\delta e^{-n\delta}\right\}
\]
and $r_\delta$ is some function of $\delta$ continuous and vanishing at zero. Then by combining (\ref{eq:BelowOnGamma3}) and (\ref{eq:BelowOnGamma4}), we obtain exactly as in (\ref{eq:contrBelow}) that there exists a constant $c_1$ independent of $n$ such that
\[
J_n \geq \sin(\delta) \int_{I_n} |b_n^2(t)| ~ |(T_l^2TQ_sw_n)(t)| d|\mes|(t) \geq c_1 \exp\{-2n(\epsilon+L\delta/2+1/\cp(\spp,\T))\}.
\]
Thus, we have that
\begin{equation}
\label{eq:ContrPart2}
\liminf_{n\to\infty} J_n^{1/2n} \geq \exp\{-\epsilon-L\delta/2-1/\cp(\spp,\T)\}.
\end{equation}
Now, by choosing $\epsilon$ and $\delta$ so small that $\epsilon+L\delta < \epsilon(\Gamma)$, we arrive at contradiction between (\ref{eq:ContrPart1}) and (\ref{eq:ContrPart2}). Therefore, the convergence in (\ref{eq:ConvIn}) holds.

Second, we show that
\begin{equation}
\label{eq:ConvPolies}
\left|\frac{\widetilde q_n^2(z) l_{s,n}(\bar z)}{q_n(z)q_n(\bar z)Q_s(z)w_n(z)}\right|^{1/2n} \cic \exp\left\{U_\D^{\ged}(z)\right\}
\end{equation}
on compact subsets of $\D\setminus\spp$. Let $K\subset\D\setminus\spp$ be compact and let $U$ be a bounded conjugate-symmetric open set containing $K$ and not intersecting $\spp$. Define
\[
b_{n,1}(z) := \prod_{\zeta\in U:~b_n(\zeta)=0}\frac{z-\zeta}{1-\bar\zeta z} \;\;\; \mbox{and} \;\;\; b_{n,2}(z) := b_n(z)/b_{n,1}(z).
\]
Further, let $q_{n,1}$ and $q_{n,2}$ be the numerators of $b_{n,1}$ and $b_{n,2}$, respectively. Corollary \ref{cor:numZeros} yields that there exists fixed $m\in\N$ such that each $b_{n,1}$ has at most $m$ zeros. Then
\[
\left|b_{n,2}(z)\right|^{1/n} \to \exp\left\{-U^{\ged}_\D(z)\right\} \;\; \mbox{and} \;\;  \left|\frac{q_{n,2}(z)}{q_{n,2}(\bar z)}\right|^{1/n} \to 1
\]
uniformly on $K$ by Theorem \ref{thm:polesCriticalPoints}. Moreover, it is an immediate consequence of the choice of $l_{s,n}$, the normality of $\{w_n\}$ in $\D$ (Lemma \ref{lem:normalWeight}), the uniform boundedness of the number of zeros of $q_{n,1}(z)q_{n,1}(\bar z)Q_s(z)$, and the lemniscate theorem \cite[Thm. 5.2.5]{Ransford} that
\[
\left|\frac{\widetilde q_{n,1}^2(z)l_{s,n}(\bar z)}{q_{n,1}(z)q_{n,1}(\bar z)Q_s(z)w_n(z)}\right|^{1/2n}\cic1, \;\;\; z\in K.
\]
Thus, we obtain (\ref{eq:ConvPolies}). It is clear now that (\ref{eq:Convergence1}) follows from (\ref{errmanip}), (\ref{eq:ConvIn}), and (\ref{eq:ConvPolies}).

Let us finally fix $p=2$ and $p=\infty$. In this former case $w_n\equiv1$ for any $n\in\N$ and in the latter $\{w_n\}$ is a normal family in $\overline\C\setminus\wspp^*$ by Lemma \ref{lem:normalWeight}, and therefore $r_n(z)$ is defined everywhere outside of $\spp\cup\spp^*$. The limit in (\ref{eq:Convergence2}) easily follows from (\ref{eq:Convergence1}), (\ref{errmanip}), (\ref{eq:ConvIn}), and (\ref{eq:ConvPolies}) since $|b_n(z)|=|b_n(1/\bar z)|^{-1}$ in $\C$. It remains only to show (\ref{eq:Convergence3}) for $p=2$ (in this case the error is defined in $\overline\C\setminus\spp$). Taking into account (\ref{eq:LimitNorms}), (\ref{eq:UpperBoundIn}), and the above-mentioned symmetry, it is sufficient to prove that
\[
\limsup_{n\to\infty}|b_n(z)|^{1/n} \leq \exp\left\{-U_\D^{\ged}(z)\right\}
\]
uniformly in $\overline\D$. The latter follows by Theorem \ref{thm:polesCriticalPoints}, the principle of descent, and the following fact. Let $a\in(0,1)$ be such that $\spp\subset\D_a$ and denote by $b_{n,a}$ the Blaschke product with those zeros of $b_n$ that are contained in $\D_a$. Then $|b_n|\leq|b_{n,a}|$ in $\D$ and the zeros of $b_{n,a}$ have the same limiting distribution ($\mu_{(\spp,\T)}$) as the zeros of $b_n$ by Corollary \ref{cor:numZeros}. Thus, $|b_{n,a}|^{1/n}$ converge locally uniformly in $\{z: a<|z|<1/a\}$ to some function that coincides with $\exp\left\{-U_\D^{\ged}(z)\right\}$ on $\{z: a<|z|\leq1\}$. This finishes the proof of the theorem.
\end{proof}

\begin{proof}[Proof of Corollary \ref{cor:BestRational}] Let $\{r_n\}$ be a sequence of solutions of $\pma(2)$ (see (\ref{eq:merApprProblem})). Then $\{r_n\}$ is a sequence of rational functions with poles in $\D$ for which (\ref{eq:Convergence3}) holds. Thus,
\[
\limsup_{n\to\infty} \rho_n(F,{\mathbb E})^{1/2n} \leq \limsup_{n\to\infty} \|F-r_n\|_\T^{1/2n} \leq \exp\left\{-\frac{1}{\cp(\spp,\T)}\right\}
\]
since $F-r_n$ are analytic on $\mathbb E$. On the other hand, AAK theory implies that
\[
\liminf_{n\to\infty}\rho_n(F,{\mathbb E})^{1/2n} \geq \liminf_{n\to\infty}\left(\inf_{g\in H_n^\infty}\|F-g\|_\infty\right)^{1/2n} = \liminf_{n\to\infty} \sigma_n(A_F)^{1/2n}.
\]
It follows from (\ref{eq:DefCrPt}) and (\ref{eq:decompHankel}) that for every $g_n$, best AAK approximant of order $n$, it holds that
\[
|F-g_n| = \sigma_n(A_F) \;\;\; \mbox{a.e. on} \;\;\; \T.
\]
Therefore, we get from (\ref{errmanip}) that
\[
\sigma_n(A_F) = \left|\frac{q_n(\xi)}{q_n(\bar\xi)}\frac{l_{s,n}(\bar\xi)B_n(\xi)}{Q_s(\xi)w_n(\xi)}\right| \;\;\; \mbox{for a.e.} \;\; \xi\in\T,
\]
where $l_{s,n}$ and $B_n$ were defined in (\ref{eq:Def1}) and (\ref{eq:Def2}), respectively. Since 
\[
\left(\frac{|q_n(z)|}{|q_n(\bar z)|}\right)^{1/n} \cic 1
\] 
on compact subsets of $\overline\C\setminus\spp$ by Theorem \ref{thm:polesCriticalPoints}, we immediately deduce from (\ref{eq:ConvIn}) and Lemma \ref{lem:normalWeight} that
\[
\lim_{n\to\infty} \sigma_n(A_F)^{1/2n} = \exp\left\{-\frac{1}{\cp(\spp,\T)}\right\},
\]
which finishes the proof of the corollary.
\end{proof}

\begin{proof}[Proof of Theorem \ref{thm:tracingPoles}] Inequality (\ref{eq:LowerM}) is trivial for any $\eta\in\rspp\cap\spp$. Suppose now that $\eta\in\rspp\setminus\spp$ and that $\underline m(\eta)< m(\eta)$. This would mean that there exists an open set $U$,  $U\cap\wspp=\{\eta\}$, such that $\underline m(\eta,U)<m(\eta)$ and therefore would exist a subsequence $\N_1\subset\N$ such that
\[\#\{\spp_n\cap U\}<m(\eta), \;\;\; n\in\N_1.\]
It was proved in Theorem \ref{thm:convergenceCapacity} that $\{g_n\}$ converges in capacity on compact subsets of $\D\setminus\spp$ to $F$. Thus, $\{g_n\}_{n\in\N_1}$ is a sequence of meromorphic functions in $U$ with at most $m(\eta)$ poles there, which converges in capacity on $U$ to a meromorphic function $\left.F\right|_{U}$ with exactly one pole of multiplicity $m(\eta)$. Then by Gonchar's lemma \cite[Lemma 1]{Gon75b} each $g_n$ has exactly $m(\eta)$ poles in $U$ and these poles converge to $\eta$. This finishes the proof of (\ref{eq:LowerM}).

Now, for any $\eta\in\rspp\setminus\spp$ the upper characteristic $\overline m(\eta)$ is finite by Corollary \ref{cor:numZeros}. Therefore there exist domains $D_\eta$, $D_\eta\cap\wspp=\{\eta\}$, such that $\overline m(\eta)=\overline m(\eta,D_\eta)$, $\eta\in\rspp\setminus\spp$. Further, let $\angs(\cdot)$ be the angle function defined in (\ref{eq:agnleSys}) for a system of $m$ intervals covering $\spp$ and let $S_n=\{\xi_{1,n},\ldots,\xi_{n,n}\}$. Then by Lemma \ref{lem:anglBound} we have
\begin{equation}
\label{eq:nameless}
\sum_{j=1}^{n}(\pi-\angs(\xi_{j,n})) \leq V(\arm)+V_\W+(m+s-1)\pi+\sum_{\zeta\in\rspp}m(\zeta)\angs(\zeta),
\end{equation}
where  $V_\W$ was defined in 
(\ref{eq:boundArgW}). The finiteness of $V_\W$ was obtained in the proof of Theorem \ref{thm:polesCriticalPoints}. Then for $n$ large enough (\ref{eq:nameless}) yields
\[\sum_{\eta\in\rspp\setminus\spp}\left(\sum_{\xi_{j,n}\in D_\eta}(\pi -\angs(\xi_{j,n}))-m(\eta)(\pi-\angs(\eta))\right) \leq V,\]
where $V$ was defined in (\ref{eq:V}). Thus,
\begin{eqnarray}
\label{eq:nmless}
\sum_{\eta\in\rspp\setminus\spp}\left(\#\{S_n\cap D_\eta\} -m(\eta)\right)(\pi-\angs(\eta)) &\leq& \sum_{\eta\in\rspp\setminus\spp}\#\{S_n\cap D_\eta\}\left(\max_{\xi\in D_\eta}\angs(\xi)-\angs(\eta)\right)  \nonumber \\
{} && + V
\end{eqnarray}
for all $n$ large enough. However, since $\{\max_{n\geq N}\#\{S_n\cap D_\eta\}\}_{N\in\N}$ is a decreasing sequence of integers, $\overline m(\eta) = \overline m(\eta,D_\eta)= \#\{S_n\cap D_\eta\}$ for infinitely many $n\in\N$. Therefore, we get from (\ref{eq:nmless}) that
\begin{equation}
\label{eq:mnless}
\sum_{\eta\in\rspp\setminus\spp}\left(\overline m(\eta) -m(\eta)\right)(\pi-\angs(\eta)) \leq V + \sum_{\eta\in\rspp\setminus\spp}\overline m(\eta)\left(\max_{\xi\in D_\eta}\angs(\xi)-\angs(\eta)\right).
\end{equation}
Observe now that the left-hand side and the first summand on the right-hand side of (\ref{eq:mnless}) are simply constants. Moreover, the second summand on the right-hand side of (\ref{eq:mnless}) can be made arbitrarily small by taking smaller neighborhoods $D_\eta$. Thus, (\ref{eq:UpperM}) follows.
\end{proof}

\section{Numerical Experiments}
\label{sec:numer}

The Hankel operator $A_F$ with symbol $F\in H^\infty+C(\T)$ is of finite rank 
if and only if $F$ is a rational function \cite[Thm. 3.11]{Partington}. In practice
one can only compute with finite rank operators, due to the necessity of 
ordering the singular values, so a preliminary rational approximation to 
$F$ is needed 
when the latter is not rational. One way to handle this problem is to 
truncate the Fourier series of $F$ at some high order $N$. This provides 
us with a rational function $F_N$ that approximates $F$ in the Wiener norm 
which, in particular, dominates any $L^p$ norm on the unit circle, 
$p\in[1,\infty]$. It was proved in \cite{HTG90} that the best approximation 
operator
from $H^\infty_n$ (mapping $F$ to $g_n$ according to (\ref{eq:operatorBestAppr})) is continuous in the Wiener norm provided $(n+1)$-st singular value of the Hankel operator is simple. It was shown in \cite[Cor. 2]{BLP00} that the 
last assertion is satisfied for Hankel operators with symbols in some
 open dense subset of $H^\infty+C(\T)$, and the same technique can be used to 
prove that it is also the case for the particular subclass (\ref{eq:mainFun}).
Thus, even though the simplicity of singular values cannot be 
asserted beforehand, it is generically true. When it prevails, one can
approximates $F_N$ instead of $F$ and get a close approximation to $g_n$ 
when $N$ is large enough. This amounts to 
perform the singular value decomposition of $A_{F_N}$
(see \cite[Ch. 16]{Young}). When $2\leq p<\infty$ there is no difficulty
with continuity issues, but the computation of $g_n$ has to rely on a
numerical search.
To numerically construct rational approximants when $p=2$, we used the
above truncation technique together with the {\it Hyperion} software 
described in \cite{rGr}.

In the numerical experiments below we approximate 
function $F$ given by the formula
\begin{eqnarray}
F(z) &=& 7\int_{[-6/7,-1/8]}\frac{e^{it}dt}{z-t} - \int_{[2/5,1/2]}\frac{3+i}{t-2i}\frac{dt}{z-t} + (2-4i)\int_{[2/3,7/8]}\frac{\ln(t)dt}{z-t} \nonumber \\
&+& \frac{2}{(z+3/7-4i/7)^2} +  \frac{6}{(z-5/9-3i/4)^3} + \frac{24}{(z+1/5+6i/7)^4}. \nonumber
\end{eqnarray}

\begin{figure}[h!]
\centering
\includegraphics[scale=.45]{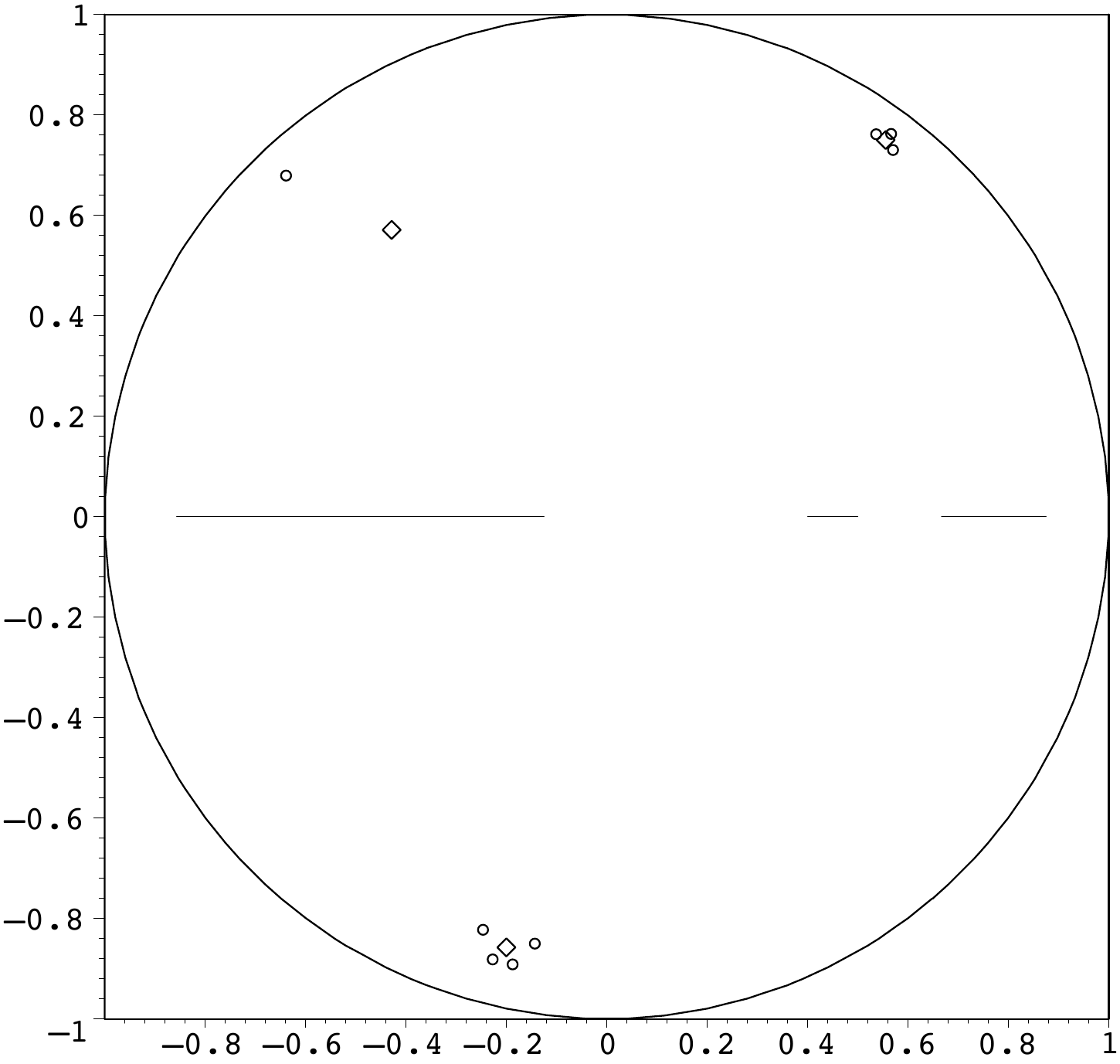}
\includegraphics[scale=.45]{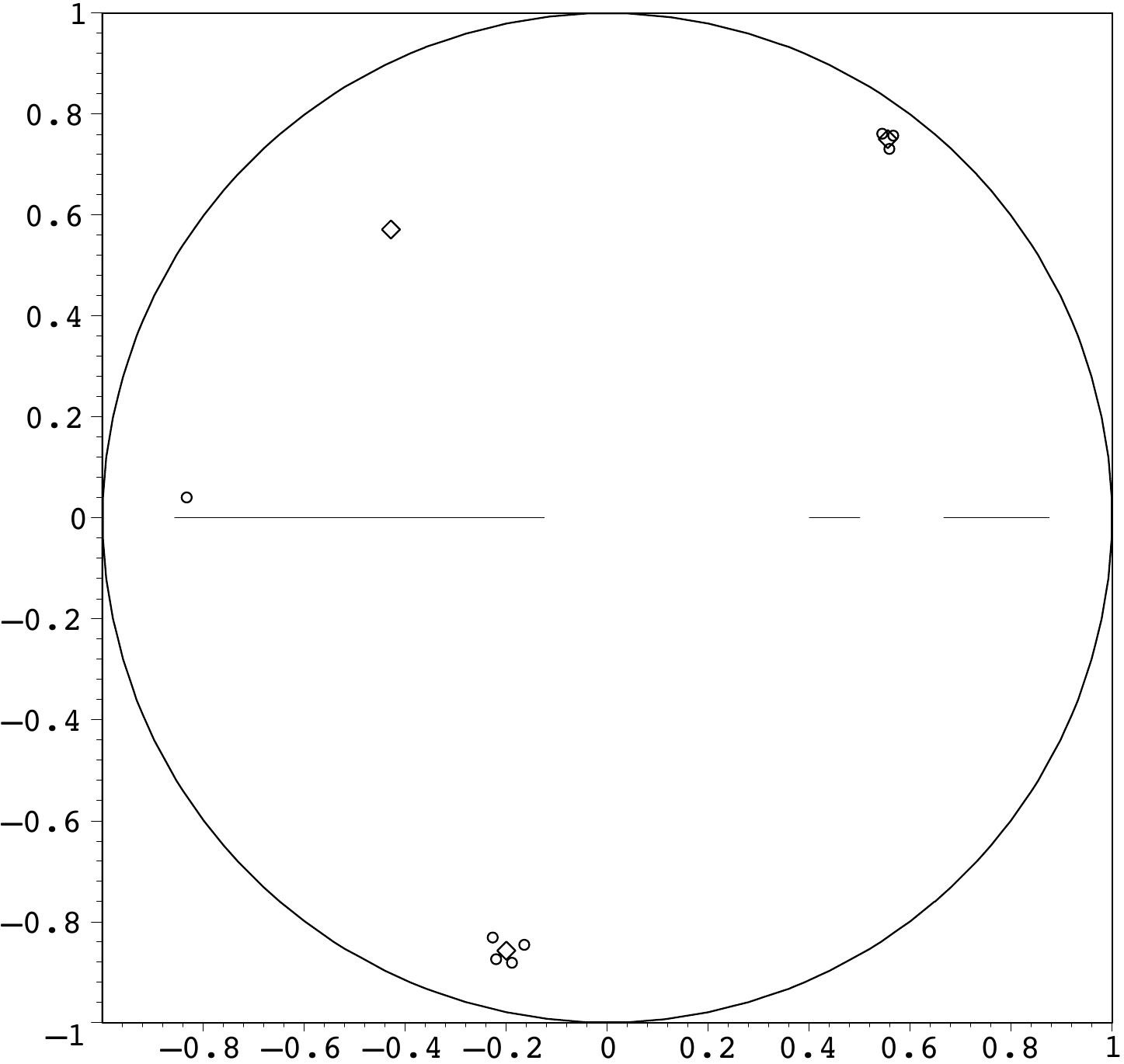}
\caption{\small AAK (left) and rational (right) approximants to $F$ of degree 8}
\end{figure}
\begin{figure}[h!]
\centering
\includegraphics[scale=.45]{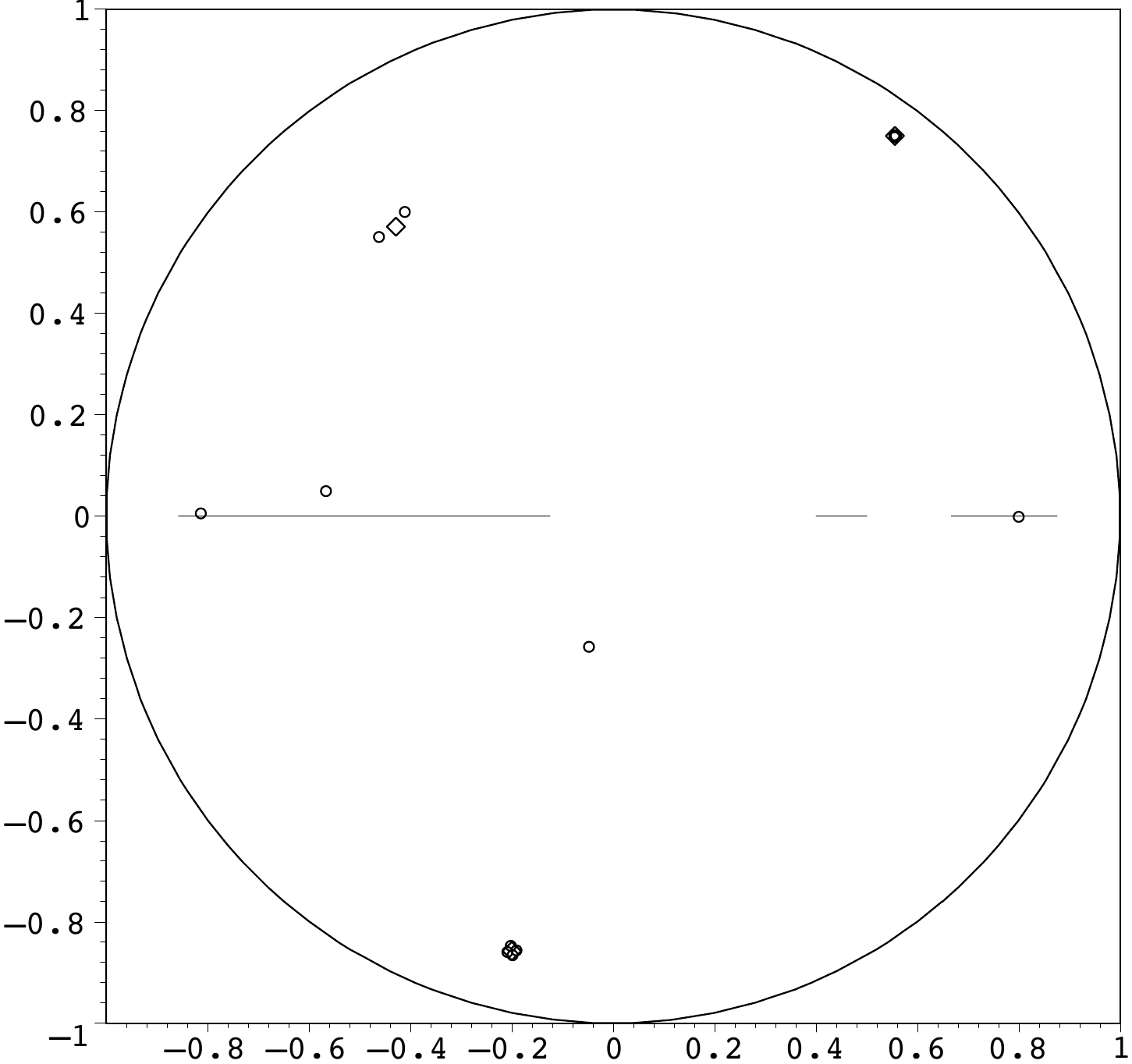}
\includegraphics[scale=.45]{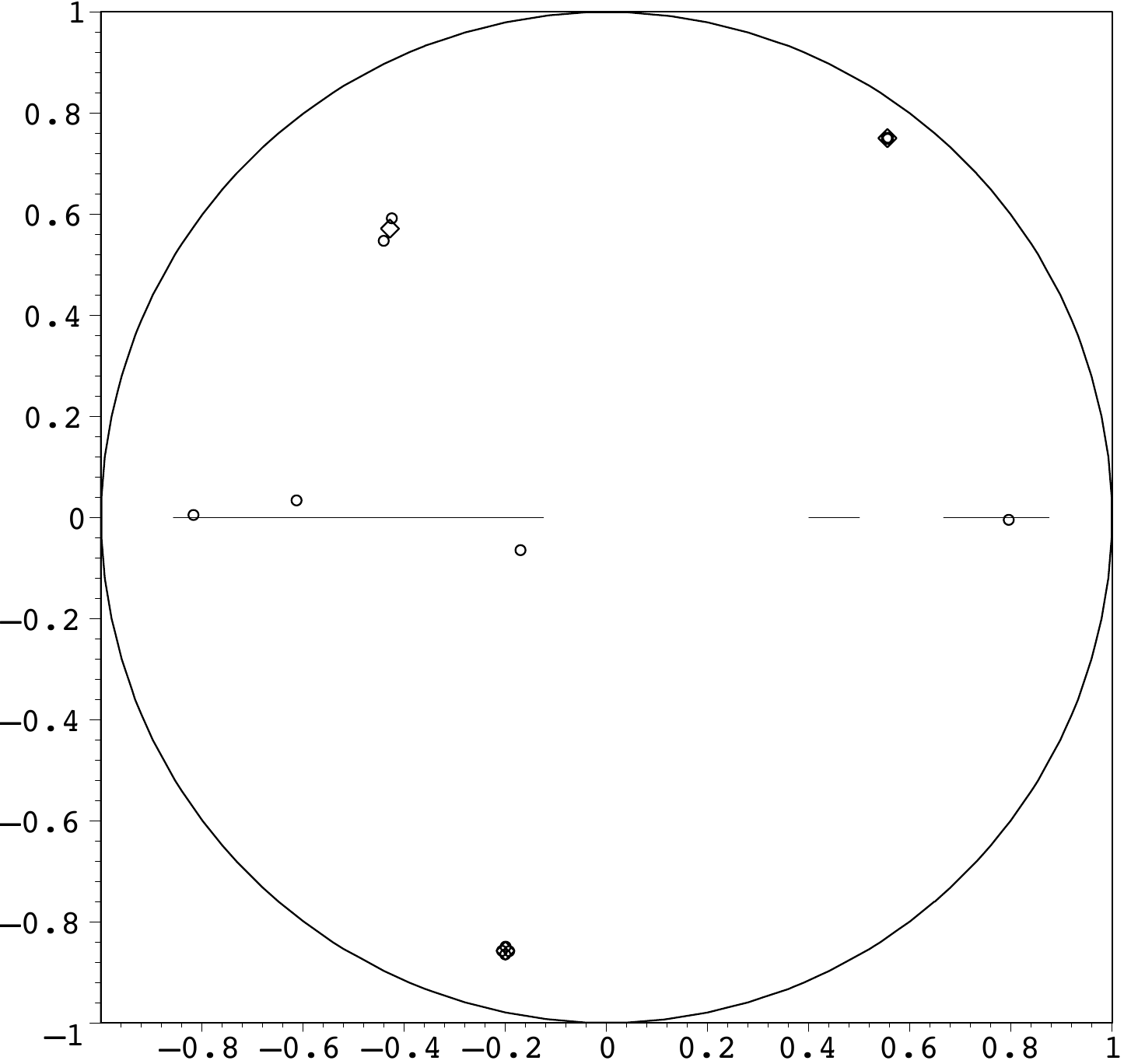}
\caption{\small AAK (left) and rational (right) approximants to $F$ of degree 13}
\end{figure}

On the figures the solid lines stand for the support of the measure, diamonds depict the polar singularities of $F$, and circles denote the poles of the corresponding approximants. Note that the poles of $F$ seem to attract the 
singularities first.

\appendix
\refstepcounter{section}
\section*{Appendix}
\label{prelnot}
\renewcommand{\theequation}{\Alph{section}.\arabic{equation}}
\renewcommand{\thesubsection}{\normalsize \Alph{section}.\arabic{subsection}}

Below we give a brief account of logarithmic potential theory that was used extensively throughout the paper. We refer the reader  to the monographs \cite{Ransford,SaffTotik} for a complete treatment.

The {\it logarithmic potential} and the {\it logarithmic energy} of a finite positive 
measure $\mu$, compactly supported in $\C$, are defined by
\begin{equation}
\label{eq1sL}
U^\mu(z):=\int\log\frac{1}{|z-t|}d\mu(t),~~~~z\in\C,
\end{equation}
and
\begin{equation}
\label{eq2sL}
I[\mu]:=\int U^\mu(z)d\mu(z)=\int\int\log\frac{1}{|z-t|}d\mu(t)d\mu(z),
\end{equation}
respectively. The function $U^\mu$ is superharmonic with values in $(-\infty,+\infty]$, which is not identically $+\infty$. It is bounded below on $\supp(\mu)$ so that $I[\mu]\in(-\infty,+\infty]$.

Let now $E\subset \C$ be compact and $\Lm(E)$ denote the set of all probability measures supported on $E$. If the logarithmic energy of every measure in $\Lm(E)$ is infinite, we say that $E$ is {\it polar}. Otherwise, there exists a unique $\mu_E\in\Lm(E)$ that minimizes the logarithmic energy over all measures in $\Lm(E)$. This measure is called the {\it equilibrium distribution} on $E$.  The {\it logarithmic capacity}, or simply the capacity, of $E$ is defined as
$$\cp(E)=\exp\{-I[\mu_E]\}.$$
By definition, the capacity of an arbitrary subset of $\C$ is the {\it supremum} of the capacities of its compact subsets. We agree that the capacity of a polar set is zero. We say that a sequence of functions $\{h_n\}$ converges {\it in capacity} to a function $h$ on a compact set $K$ if for any $\epsilon>0$ it holds that
\[
\lim_{n\to\infty} \cp\left(\left\{z\in K:~ |h_n(z)-h(z)| \geq \epsilon\right\}\right) = 0.
\]

Another important concept is the \emph{regularity} of a compact set. We restrict to the case when $E$ has connected complement. Let $g_{\overline\C\setminus E}(\cdot,t)$ be the Green function of $\overline\C\setminus E$ with pole at $t\in\overline\C\setminus E$, i.e. the unique function such that
\begin{itemize}
 \item [(i)]   $g_{\overline\C\setminus E}(z,t)$ is a positive harmonic function in $\left(\overline\C\setminus E\right)\setminus\{t\}$, which is bounded outside each neighborhood of $t$;
 \item [(ii)]  $\displaystyle g_{\overline\C\setminus E}(z,t) - \left\{\begin{array}{ll}\log|z|, & t=\infty, \\ -\log|z-t|, & t\neq\infty, \end{array}\right.$ is bounded near $t$;
 \item [(iii)] $\displaystyle \lim_{z\to\xi, \; z\in D}g_{\overline\C\setminus E}(z,t)=0$ for quasi every $\xi\in E$.
\end{itemize} 
Points of continuity of $g_{\overline\C\setminus E}(\cdot,t)$ on $\partial_e E$, the outer boundary of $E$, are called {\it regular}, other points on $\partial_e E$ are called irregular; the latter form a polar set. If every point of $\partial_e E$ is regular, we say that the whole set $E$ is regular.

Throughout we use the concept of {\it balayage} of a measure
(\cite[Sec. II.4]{SaffTotik}). Let $D$ be a domain (connected open set) with compact boundary $\partial
D$ whose complement has positive capacity, and $\mu$ be a finite Borel measure 
with compact support in $D$. 
Then there exists a unique Borel measure $\widehat\mu$ supported on 
$\partial D$, with total mass is equal to that of $\mu$: 
$\|\mu\|=\|\widehat\mu\|$, whose potential $U^{\widehat\mu}$ is bounded
on $\partial D$ and satisfies for some constant
$c(\mu;D)$
\begin{equation}
\label{eq:equalBal}
U^{\widehat\mu}(z) =    U^\mu(z)+c(\mu;D) 
\mbox{~~~~for q.e.} \;\;\; z\in\C\setminus D.
\end{equation}
Necessarily then, we have that $c(\mu;D)=0$ if $D$ is bounded and
$c(\mu;D)=\int g_D(t,\infty)d\mu(t)$ otherwise. 
Equality in (\ref{eq:equalBal}) holds for all $z\in\C\setminus\overline D$ 
and also at all regular points of $\partial D$. 
The measure $\widehat\mu$ is called the balayage of $\mu$ onto $\partial D$.
It has the property that
\begin{equation}
\label{bornebalayage}
U^{\widehat\mu}(z) \leq U^\mu(z)+c(\mu;D) 
\mbox{~~~~for every}  \;\;\; z\in\C,
\end{equation}
and also that
\begin{equation}
\label{balayageh}
\int h\,d\mu=\int h\,{d\widehat\mu}
\end{equation}
for any function $h$ which is harmonic in $D$ and continuous in $\overline{D}$
(including at infinity if $D$ is unbounded). From its defining properties
$\widehat\mu$ has finite energy, therefore it cannot charge polar sets. 
Consequently, on solving the generalized Dirichlet problem 
\cite[Thm. 4.1.5]{Ransford} for an arbitrary positive continuous function on 
$\partial D$, it follows from (\ref{balayageh}) that the balayage of a probability measure is a
probability measure. 

The minimal energy problem can also be formulated for signed measures
\cite[Thm. VIII.1.4]{SaffTotik}. In particular for
$E_1$, $E_2$ two disjoint compact sets of positive
capacity, there exists a unique measure $\mu^*=\mu_1^*-\mu_2^*$, with 
$\mu_1^*\in\Lm(E_1)$ and $\mu_2^*\in\Lm(E_2)$, that minimizes the energy 
integral
\begin{equation}
\label{signen}
I[\mu_1-\mu_2]=\int\log\frac{1}{|z-t|}d(\mu_1-\mu_2)(t)d(\mu_1-\mu_2)(z), \;\;\; \mu_j\in\Lm(E_j), \;\;\; j=1,2.\end{equation}
It can be proved (\cite[Lemma 1.8]{SaffTotik}) that $I[\mu^*]$ is 
positive and finite. The value $\cp(E_1,E_2)=1/I[\mu^*]$ is called the {\it
  condenser capacity} of the pair $(E_1,E_2)$. Further, it holds that
$\mu_1=\widehat\mu_2$
and $\mu_2=\widehat\mu_1$, where $\widehat\mu_1$ (resp. $\widehat\mu_2$)
indicates the balayage of $\mu_1$ (resp. $\mu_2$) onto 
$\partial(\overline{\C}\setminus E_2)$ (resp. $\partial(\overline{\C}\setminus
E_1)$); this property in fact characterizes $\mu^*$, see
\cite[Thm. VIII.2.6]{SaffTotik}. 

In analogy to the logarithmic case, one can define the {\it Green potential}
and the {\it Green energy} of a positive measure $\mu$ supported 
in a domain $D$ with compact non-polar boundary. 
The only difference is now that, in (\ref{eq1sL})-(\ref{eq2sL}), the
logarithmic kernel $\log(1/|z-t|)$ gets replaced by 
$g_D(z,t)$, the Green function for $D$ with pole at $t\in D$.  The Green potential relative to the domain $D$ of a finite
positive measure $\mu$ compactly supported in $D$ is given by
\[U_D^\mu(z)=\int g_D(z,t)\,d\mu(t).\]
It can be re-expressed in terms of the logarithmic potentials
of $\mu$ and of its balayage $\widehat\mu$ onto $\partial D$ by the formula
\cite[Thm. II.4.7 and Thm. II.5.1]{SaffTotik}
\begin{equation}
\label{eq:toRemind}
U^{\widehat\mu-\mu}(z) = c(\mu;D) - U_D^\mu(z), \;\;\; z\in D,
\end{equation}
where $c(\mu;D)$ was defined after equation (\ref{eq:equalBal}).
Moreover, (\ref{eq:toRemind}) continues to hold at every regular point of
$\partial D$; in particular, it holds q.e. on $\partial D$.

Exactly as in the logarithmic case, if $E$ is a compact nonpolar subset of
$D$,
there exists a unique measure $\mu_{(E,\partial D)}\in\Lm(E)$ that minimizes the
Green energy among all measures in $\Lm(E)$. This measure is called the {\it
  Green equilibrium distribution} on $E$ relative to $D$. 
By (\ref{eq:toRemind}) we have that
$$U_D^{\mu_{(E,\partial D)}}(z)=U^{\mu_{(E,\partial D)}}(z)-U^{\widehat{\mu_{(E,\partial D)}}}(z)+c(\mu_{(E,\partial D)};D),
~~~~z\in D,~~{\rm and~q.e.~} z\in\partial D,$$
where $\widehat{\mu_{(E,\partial D)}}$ is the balayage of $\mu_{(E,\partial D)}$ onto $\partial D$.
In addition, the Green equilibrium distribution satisfies
\begin{equation}
\label{eq:GreenEqual}
U_D^{\mu_{(E,\partial D)}}(z)=\frac{1}{\cp(E,\partial D)}, \;\;\; \mbox{for q.e.} \;\;\; z\in E,
\end{equation}
where $\cp(E,\partial D)$ is {\it Green (condenser) capacity} of $E$ relative to $D$ which is the reciprocal of the minimal Green energy among all measures in $\Lm(E)$. Moreover, equality in (\ref{eq:GreenEqual}) holds at all regular points of $E$.

For the reader's convenience, we formulate below a proposition that was of
particular use to us. It has to do with the specific geometry 
of the disk, and we could not find an appropriate
reference for it in the literature. The proof of this proposition can be found in \cite[Prop. A.1]{thYat}.
\newline
\newline
{\bf Proposition A} {\it Let $E\subset\D$ be a compact set of positive
  capacity not containing $0$ with connected complement, 
and $E^*$ stand for its reflection across the unit circle,
i.e. $E^*:=\{z\in\C: \; 1/\bar z\in E\}$. Further, let $\mu\in\Lm(E)$ and
$\sigma\in\Lm(E^*)$ solve the signed energy problem for the condenser 
$(E,E^*)$. Then, we have that
\begin{itemize}
\item[(a)] $\sigma$ is reflected from $\mu$ across the unit circle,
i.e. $\sigma(B)=\mu(B^*)$ for any Borel set $B$, and likewise $\mu$ is
reflected from $\sigma$;        
\item[(b)] $\mu$ is the Green equilibrium distribution on $E$ relative to $\overline\C\setminus E^*$ and $\sigma$ is the Green equilibrium distribution on $E^*$ relative to $\overline\C\setminus E$;
        \item[(c)] $\widetilde\mu=\widetilde\sigma$, where $\widetilde\lambda$ denotes the balayage of the measure $\lambda$ on $\T$. Moreover, the balayage of $\widetilde\mu$ onto $E$ is $\mu$ and the balayage of $\widetilde\mu$ onto $E^*$ is $\sigma$;
        \item[(d)] $\widetilde\mu$ is the Green equilibrium distribution on
          $\T$ relative to both $\overline\C\setminus E$ and $\overline\C\setminus E^*$;
        \item[(e)] $\mu$ is the Green equilibrium distribution on $E$ relative to $\D$ and $\sigma$ is the Green equilibrium distribution on $E^*$ relative to $\C\setminus\overline\D$.
\end{itemize}}

\bibliographystyle{plain}
\small
\bibliography{aak}

\end{document}